\documentclass[11pt]{article}
\pdfoutput=1

\def\pb{\pagebreak}   

\usepackage{lineno,lmodern} 

\usepackage[dvipsnames]{xcolor}

\usepackage{comment}

\def\beq{\begin{equation} }\def\eeq{\end{equation} }\def\1{\mathbf{1}}

\usepackage[framemethod=default]{mdframed}
\usepackage{caption}

\usepackage{indentfirst}
\usepackage{bm, mathrsfs, graphics,float,amssymb,amsmath,subeqnarray,setspace,graphicx,amsthm,epstopdf,subfigure, enumerate, color,booktabs}
\usepackage[utf8]{inputenc}
\usepackage[colorlinks,
linkcolor=red,
anchorcolor=blue,
citecolor=blue
]{hyperref}
\usepackage{natbib}

\usepackage{fullpage}
\numberwithin{equation}{section}

\newtheorem{lemma}{Lemma}
\newtheorem{theorem}{Theorem}
\newtheorem{proposition}{Proposition}
\newtheorem{definition}{Definition}
\newtheorem{corollary}[theorem]{Corollary}
\newtheorem{remark}{Remark}

\ifx\assumption\undefined
\newtheorem{assumption}{Assumption}
\fi

\newcommand{\cO}{\mathcal{O}}

\newcommand{\EE}{\mathbb{E}}
\newcommand{\RR}{\mathbb{R}}

\newcommand{\tx}{\tilde{x}}
\newcommand{\cN}{\mathcal{N}}
\newcommand{\cS}{\mathcal{S}}

\usepackage{bbm}

\newcommand{\cB}{\mathcal{B}}

\newcommand{\tg}{\tilde{g}}

\newcommand{\tu}{\tilde{u}}

\usepackage{multirow}
\usepackage{tablefootnote}
\usepackage{colortbl}
\usepackage{hhline}

\usepackage{algorithm}
\usepackage{algorithmic}

\newcommand{\norm}[1]{\left\|#1\right\|}
\newcommand{\dotprod}[1]{\left\langle #1\right\rangle}

\newcommand{\xiangyu}[1]{{{\textcolor{red}{[xiangyu: #1]}}}}

\begin{document}
\title{
Can a One-Point Feedback Zeroth-order Algorithm Achieve Linear Dimension Dependent Sample Complexity?
}

\author{
Haishan Ye
\thanks{
Xi'an Jiaotong University;
email: yehaishan@xjtu.edu.cn
}
\and
Xiangyu Chang
\thanks{
Xi'an Jiaotong University;
email: xiangyuchang@xjtu.edu.cn
}
}
\date{\today}

\maketitle

\begin{abstract}
We revisit the one-point feedback zeroth-order (ZO) optimization problem, a classical setting in derivative-free optimization where only a single noisy function evaluation is available per query. 
Compared to their two-point counterparts, existing one-point feedback ZO algorithms typically suffer from poor dimension dependence in their sample complexities---often quadratic or worse---even for convex problems. 
This gap has led to the open question of whether one-point feedback ZO algorithms can match the optimal \emph{linear} dimension dependence achieved by two-point methods.

In this work, we answer this question \emph{affirmatively}. 
By leveraging uniform ball smoothing and a two-stage convergence analysis framework that first establishes rates for the smoothed objective before transferring them to the original function, we show that the one-point residual feedback ZO algorithm proposed by~\citet{zhang2022new} can achieve optimal sample complexity in several settings. 
For convex Lipschitz objective functions in deterministic setting, our bound matches the $\mathcal{O}(d\varepsilon^{-2})$ complexity of the optimal two-point feedback method; for convex smooth functions, we obtain $\mathcal{O}(d\varepsilon^{-3/2})$; and for nonconvex smooth functions, we achieve linear dimension dependence with improved constants over prior work. 
In the stochastic setting, our analysis yields bounds with quadratic dependence on $d$ but significantly reduced dependence on the value variance compared to~\citet{zhang2022new}, leading to sharper guarantees in practical regimes. 
Moreover, in the important and widely studied least-squares setting, our framework further achieves linear dependence on the dimension, thereby substantially strengthening the practical impact of our results.
Our results not only close a long-standing gap in the theory of one-point feedback ZO optimization, but also demonstrate that careful analysis can unlock optimality without modifying the underlying algorithm.

\end{abstract}

\section{Introduction}

This paper considers both deterministic and stochastic formulations of unconstrained optimization problems. We begin by considering the deterministic case, where the objective is to minimize a function over the entire Euclidean space:
\begin{equation} \label{eq:prob}
\min_{x \in \RR^d} f(x),
\end{equation}
where $x \in \RR^d$ denotes the decision variable, and $f: \RR^d \to \RR$ is a real-valued objective function that is assumed to be deterministic and well-defined over the domain.

In many practical applications~\citep{malik2020derivative,he2023model,malladi2023fine}, the objective function depends on uncertain data or stochastic parameters. In such cases, the optimization problem \eqref{eq:prob} is naturally formulated as a stochastic program, where the objective function $f$ is expressed as an expectation over a random variable $\xi$:
\begin{equation} \label{eq:prob_s}
f(x) = \EE_{\xi}\left[ f(x, \xi) \right],
\end{equation}
where $\xi$ is a random variable representing inherent randomness in the environment, such as data noise, model uncertainty, or fluctuating system parameters. 
Due to the expectation operator in Eq.~\eqref{eq:prob_s}, the exact computation of $f(x)$ is often intractable in practice. Instead, one can typically access only stochastic approximations of the objective by sampling from the distribution of $\xi$. That is, for a given $x$, one observes only noisy evaluations $f(x, \xi)$, where $\xi$ is drawn from its underlying distribution~\citep{berahas2022theoretical}.

One straightforward approach to solve~\eqref{eq:prob} is the (stochastic) gradient descent method \citep{nesterov2013introductory,bottou2012stochastic}. 
These methods rely on the availability of first-order information, namely the gradient of the objective function with respect to the decision variable. However, in many practical scenarios, such gradient information is unavailable or prohibitively expensive to compute, rendering traditional gradient-based methods inapplicable.
For example, in the context of black-box adversarial attacks against deep neural networks, the attacker does not have access to the internal gradients of the model. Instead, only the inputs and outputs of the model—viewed as a black-box function—are observable~\citep{ilyas2018black}. 
In such a setting, zeroth-order (ZO) optimization techniques~\citep{ghadimi2013stochastic,nesterov2017random,berahas2022theoretical}, which rely solely on function value evaluations, offer a powerful alternative. 
These methods have received significant attention in a variety of domains where gradient information is either inaccessible or costly to compute. Applications include fine-tuning large language models~\citep{malladi2023fine}, solving reinforcement learning problems~\citep{malik2020derivative,mania2018simple,salimans2017evolution}, and controlling complex physical systems~\citep{chen2021safe,luo2020socially,he2023model}.

According to the number of function evaluations queried per iteration, ZO optimization can be broadly classified into two categories: \emph{one-point feedback} and \emph{two-point feedback} methods. 
In the one-point feedback setting, the algorithm accesses only one function value at each iteration and uses this value to construct a stochastic approximation of the gradient.
\citet{spallone} proposes the first one-point feedback ZO algorithm, named SPSA1, for solving~\eqref{eq:prob}. 
The gradient estimator is constructed as
\begin{equation} \label{eq:G1}
G(x) = \frac{f(x + \alpha u)}{\alpha} u^{-1}, \quad \alpha \in \RR,
\end{equation}
where $u \in \{-1, +1\}^d$ is a random vector whose entries are independently sampled from the Rademacher distribution~\citep{sadegh1998optimal}, and $u^{-1}$ denotes the element-wise inverse of $u$.
Given that $u^{-1} = u$ for all $u \in \{-1, +1\}^d$, the gradient estimator simplifies to
\begin{equation} \label{eq:G2}
G(x) = \frac{f(x + \alpha u)}{\alpha} u, \quad u \in \{-1, +1\}^d.
\end{equation}
For stochastic optimization problems as defined in Eq.~\eqref{eq:prob_s}, \citet{flaxman2004online} propose the first one-point feedback ZO algorithm under the online bandit optimization framework with the gradient estimator as
\begin{equation}\label{eq:G}
G(x, \xi) = d\cdot\frac{f(x+\alpha u, \xi)}{\alpha} u \mbox{ with } \alpha\in\RR\mbox{ and } u\sim\cS(0, I_d),
\end{equation}  
where $u\sim \cS(0, I_d)$ means uniformly sampling a random vector from the sphere of a unit ball centered at zero.



Unfortunately, the gradient estimators in Eq.~\eqref{eq:G1} and Eq.~\eqref{eq:G} suffer from high variance. 
For instance, the high variance of $G(x)$ defined in Eq.~\eqref{eq:G1} can be illustrated by the following derivation:
\begin{equation} \label{eq:G_var}
\begin{aligned}
\EE\left[\norm{G(x)}^2\right] 
=& 
\EE\left[ \norm{ \frac{f(x) + \alpha \langle \nabla f(x), u \rangle + \cO(\alpha^2)}{\alpha} u }^2 \right] \\
=&
\EE\left[ \norm{ \frac{f(x)}{\alpha} u + uu^\top \nabla f(x) + \cO(\alpha u) }^2 \right] \\
\le&
3 \EE\left[ \left\| \frac{f(x)}{\alpha} \right\|^2 \cdot \norm{u}^2 \right]
+ 3 \EE\left[ \norm{uu^\top \nabla f(x)}^2 \right]
+ 3 \EE\left[ \norm{ \cO(\alpha) u }^2 \right] \\
=& 3d \left( \frac{f(x)}{\alpha} \right)^2 + 3d \norm{\nabla f(x)}^2 + 3d \cO(\alpha^2),
\end{aligned}
\end{equation}
where the final equality follows from the fact that $u \in \{-1, +1\}^d$ and $\EE[\norm{u}^2] = d$. Since $\alpha$ is typically chosen to be a small value dependent on the target precision, the term $\left( \frac{f(x)}{\alpha} \right)^2$ becomes dominant and leads to large variance. 
Likewise, the stochastic one-point gradient estimator in Eq.~\eqref{eq:G} also suffers from high variance due to a similar scaling behavior. As a result, the ZO algorithm proposed by \citet{flaxman2004online} suffers from poor sample complexity, with a convergence rate depending on $\varepsilon^{-4}$ and $d^4$, where $\varepsilon$ denotes the target optimization precision.

In contrast to the one-point feedback estimator in Eq.~\eqref{eq:G}, two-point feedback ZO algorithms can significantly reduce the variance by utilizing an additional function evaluation per iteration. A widely used two-point feedback gradient estimator is given by~\citet{shamir2017optimal}:
\begin{equation} \label{eq:G_p}
G'(x) = d \cdot \frac{f(x + \alpha u) - f(x)}{\alpha} u, \quad \alpha \in \RR,\quad u \sim \cS(0, I_d).
\end{equation}
The variance of $G'(x)$ can then be upper bounded as follows:
\begin{align*}
\EE\left[\norm{G'(x)}^2\right]
=& \EE\left[ \norm{d \cdot \frac{f(x + \alpha u) - f(x)}{\alpha} u}^2 \right] \\
=& \EE\left[ \norm{d \cdot \left( \langle \nabla f(x), u \rangle + \cO(\alpha) \right) u }^2 \right] \\
=& \EE\left[ \norm{d \cdot uu^\top \nabla f(x) + \cO(\alpha) \cdot d u}^2 \right] \\
\leq& 2d^2 \EE\left[\norm{uu^\top \nabla f(x)}^2\right] + 2d^2 \EE\left[\norm{\cO(\alpha) u}^2\right] \\
=& 2d \norm{\nabla f(x)}^2 + 2d^2 \cdot \cO(\alpha^2),
\end{align*}
where the final equality follows from Lemma \ref{lem:uua} and $\EE[\norm{u}^2] = 1$.
Compared to the variance of $G(x)$ in Eq.~\eqref{eq:G_var}, the variance of $G'(x)$ is nearly independent of the choice of $\alpha$, particularly when $\alpha$ is small. As a result, two-point feedback ZO algorithms can achieve an order-wise faster convergence rate than their one-point feedback counterparts~\citep{nesterov2017random,shamir2017optimal,ghadimi2013stochastic}.


Due to the high variance of one-point feedback gradient estimators, achieving an $\varepsilon$-suboptimal solution of \eqref{eq:prob} typically requires sample complexity with heavy dependence on the problem dimension. For instance, the sample complexity of the algorithm in~\citet{flaxman2004online} scales at least as $d^4$.
In contrast, two-point feedback ZO algorithms with tight convergence guarantees typically exhibit linear dependence on $d$~\citep{shamir2017optimal}. As a result, \citet{larson2019derivative} raise concerns about the practicality of one-point feedback ZO methods in high-dimensional settings.

To overcome the limitations of one-point feedback ZO optimization, numerous efforts have been made to reduce the variance of one-point feedback gradient estimators in order to accelerate convergence and, in particular, to mitigate the dimensional dependence in sample complexity~\citep{saha2011improved,dekel2015bandit,gasnikov2017stochastic,zhang2022new,chen2022improve}. One effective strategy is to leverage the high-order smoothness of the objective function~\citep{akhavan2020exploiting,bach2016highly}.
Recently, inspired by two-point feedback ZO methods, \citet{zhang2022new} propose the \emph{one-point residual feedback ZO algorithm}, which utilizes information from the previous iteration to reduce the variance of the gradient estimator constructed using one-point feedback oracles. Specifically, for deterministic optimization problems, they introduce the following estimator:
\begin{equation} \label{eq:gt_def}
g_t =  \frac{f(x_t + \alpha u_t) - f(x_{t-1} + \alpha u_{t-1})}{\alpha} u_t, \quad \alpha \in \RR,\quad u_t \sim \cN(0, I_d),
\end{equation}
where $\cN(0, I_d)$ is the $d$-dimension standard Gaussian distribution. 
For stochastic optimization problems, the corresponding estimator is
\begin{equation} \label{eq:sgt_def}
\tg_t =  \frac{f(x_t + \alpha u_t, \xi_t) - f(x_{t-1} + \alpha u_{t-1}, \xi_{t-1})}{\alpha} u_t, \quad u_t \sim \cN(0, I_d).
\end{equation}
The estimator in Eq.~\eqref{eq:gt_def} closely resembles the two-point estimator in Eq.~\eqref{eq:G_p}. When $f(x)$ is Lipschitz continuous, the value of $f(x_{t-1} + \alpha u_{t-1})$ tends to be close to $f(x_t)$, thereby significantly reducing the variance compared to the standard one-point estimator in Eq.~\eqref{eq:G2}.
Benefiting from this variance reduction, \citet{zhang2022new} show that the one-point residual feedback ZO algorithm achieves improved sample complexity over existing methods. Furthermore, for smooth objective functions, \citet{chen2022improve} introduce a momentum-based scheme to further enhance the efficiency of one-point feedback ZO algorithms by reducing their sample complexity.

\begin{table*}
\begin{center}
	\scriptsize
	\begin{tabular}{ccccccc}
		\hline 
		Methods	& Methods & cvx, $L_0$-Lip &  $\mu$-cvx, $L_0$-Lip   & cvx, $L$-smooth & $\mu$-cvx, $L$-smooth & noncvx $L$-smooth\\
		\hline\addlinespace
		two-point	&	\citet{nesterov2017random}  & $d^2\varepsilon^{-2}$ & - & $d\varepsilon^{-1}$ & $d \mu^{-1}\log\frac{1}{\varepsilon}$  & $d\varepsilon^{-2}$ \\\addlinespace
		two-point	&	\citet{rando2023optimal}  & $d\varepsilon^{-2}$ & - & $d\varepsilon^{-1}$ & - & $d\varepsilon^{-2}$ \\\addlinespace
		\hline\addlinespace
		one-point   & \citet{zhang2022new} & $ d^2\varepsilon^{-2}$ & - & $d^3\varepsilon^{-3/2}$ & - & $d^3\varepsilon^{-3}$
		\\\addlinespace
		one-point  & \citet{chen2022improve} & - & - & $d^{3/2} \varepsilon^{-3/2}$  & - & $d^{3/2}\varepsilon^{-3}$
		\\\addlinespace
		\hline\addlinespace
		one-point & \textbf{Our result} & $d\varepsilon^{-2}$ & $ d\mu^{-1}\varepsilon^{-1}$ & $d\varepsilon^{-3/2}$ & $d\mu^{-1}\varepsilon^{-1/2}\log\frac{1}{\varepsilon}$ & $d\varepsilon^{-3}$
		\\\addlinespace
		\hline\addlinespace
	\end{tabular}
\end{center}
\caption{Comparing our sample complexities with existing works in different settings, for the deterministic optimization problem. 
Notations ``cvx'' and ``$\mu$-cvx'' denote ``convex'' and ``$\mu$-strongly convex''. 
Notation ``$L_0$-Lip'' denotes $L_0$-Lipschitz. 
The detailed definitions of these notions are listed in Section~\ref{sec:prelim}. }\label{tb:det}
\vskip -0.2in
\end{table*}

Although recent studies have achieved certain progress in reducing the sample complexity of one-point feedback ZO algorithms, these methods still exhibit relatively high sample complexities. For example, in the deterministic setting with a convex and Lipschitz continuous objective function, the best known sample complexity is $\cO\left(\frac{d^2}{\varepsilon^2}\right)$~\citep{zhang2022new}, which exhibits quadratic dependence on the dimension $d$, as opposed to the preferred linear scaling.
\textbf{Whether a one-point feedback ZO algorithm can achieve sample complexity with linear dependence on $d$ remains an open question in the literature.}
In this work, we provide an \emph{affirmative} answer to this open problem. Interestingly, our result does not require introducing a new one-point feedback ZO algorithm. We demonstrate that, through a refined convergence analysis framework and using the uniform ball distribution instead of the Gaussian distribution used in \citet{zhang2022new}, the one-point residual feedback ZO algorithm proposed by~\citet{zhang2022new} can attain significantly improved sample complexity bounds.
The improved sample complexity results are summarized in Table~\ref{tb:det} and Table~\ref{table2}, for deterministic and stochastic optimization problems, respectively.

\begin{table*}
\begin{center}
	\scriptsize
	\begin{tabular}{ccccccc}
		\hline 
		Methods	& Methods & cvx, $L_0$-Lip &  $\mu$-cvx, $L_0$-Lip   & cvx, $L$-smooth & $\mu$-cvx, $L$-smooth & noncvx $L$-smooth \\
		\hline\addlinespace
		two-point	&	\citet{duchi2015optimal} & $d \varepsilon^{-2}\log d$ & - & $d\varepsilon^{-2}$ &- & - \\\addlinespace
		two-point	&\citet{shamir2017optimal} & $d\varepsilon^{-2}$ & - & - &  - & -  \\\addlinespace
		\hline\addlinespace
		one-point & \citet{flaxman2004online} & $d^4\varepsilon^{-4}$ & - & - & - & - \\\addlinespace
		one-point	&	\citet{gasnikov2017stochastic} & $d^2\varepsilon^{-4}$ & $d^2\mu^{-1}\varepsilon^{-3}$ & $d^2\varepsilon^{-3}$& $d^2\mu^{-1} \varepsilon^{-2}$ & - \\\addlinespace
		one-point   & \citet{zhang2022new} & $ d^2\sigma_0^8\varepsilon^{-4}$ & - & $d^2\sigma_0^6\varepsilon^{-3}$ & - & $d^4 \varepsilon^{-6}$
		\\\addlinespace
		\hline\addlinespace
		one-point & \textbf{Our result} & $ d^2\sigma_0^2\varepsilon^{-4}$ & $ d^2\sigma_0^2\mu^{-1}\varepsilon^{-3}$ & $d^2\sigma_0^2\varepsilon^{-3}$ & $d^2\sigma_0^2 \mu^{-2} \varepsilon^{-2}$ & $d^2\sigma_0^2\varepsilon^{-6}$
		\\\addlinespace
	
        one-point & \textbf{Our result + Proposition \ref{prop:var}} & $ d\sigma_1^2\varepsilon^{-4}$ & $ d\sigma_1^2\mu^{-1}\varepsilon^{-3}$ & $d\sigma_1^2\varepsilon^{-3}$ & $d\sigma_1^2 \mu^{-2} \varepsilon^{-2}$ & $d\sigma_1^2\varepsilon^{-6}$
		\\\addlinespace
		\hline\addlinespace
	\end{tabular}
\end{center}
\caption{Comparing our sample complexities with existing works in different settings for the stochastic optimization problem.
Notation  $\sigma_0^2$ and $\sigma_1^2$ are defined in Section~\ref{sec:prelim}.   }\label{table2}
\vskip -0.2in
\end{table*}

In the deterministic setting, when the objective function is convex and $L_0$-Lipschitz continuous, \citet{rando2023optimal} provide an optimal two-point feedback ZO algorithm with a sample complexity of $\cO\left(\frac{d}{\varepsilon^2}\right)$. In contrast, the best known one-point feedback ZO algorithm, proposed by~\citet{zhang2022new}, achieves a sample complexity of only $\cO\left(\frac{d^2}{\varepsilon^2}\right)$, revealing a clear gap between one-point and two-point methods under this setting.
Our work closes this gap. We show that the one-point residual feedback ZO algorithm of~\citet{zhang2022new}, when combined with a refined convergence analysis, can also attain the optimal sample complexity of $\cO\left(\frac{d}{\varepsilon^2}\right)$. To the best of our knowledge, this is the first result that demonstrates a one-point feedback ZO algorithm matching the performance of the best known two-point method in terms of dimension dependence. Therefore, the one-point residual feedback ZO algorithm is optimal for convex and $L_0$-Lipschitz functions.

We further analyze the performance of the one-point residual feedback ZO algorithm under various standard smoothness assumptions:
\begin{itemize}
  \item \textbf{Convex and $L$-smooth functions:} We establish a sample complexity of $\cO(d \varepsilon^{-3/2})$, which retains linear dependence on the dimension $d$ and improves significantly upon previous one-point feedback ZO algorithms~\citep{zhang2022new,chen2022improve}.
  \item \textbf{Nonconvex and $L$-smooth functions:} Our analysis yields a sample complexity that is again linearly dependent on $d$, and strictly better than those of~\citet{zhang2022new,chen2022improve}.
  \item \textbf{$\mu$-strongly convex and $L$-smooth functions:} We show that the sample complexity can be further improved due to the stronger structure of the objective function.
\end{itemize}

In the stochastic optimization setting, Table~\ref{table2} demonstrates that one-point feedback ZO algorithms generally require higher sample complexity than in the deterministic case. This discrepancy arises from the intrinsic sampling noise, which is significantly amplified when constructing one-point feedback gradient estimators. Unlike the deterministic setting, where variance can be reduced using information from previous iterates (as in the one-point residual feedback ZO algorithm), such variance reduction techniques are largely ineffective in the stochastic setting.
A notable example is the work of~\citet{gasnikov2017stochastic}, whose reported sample complexities appear comparable to ours. However, their analysis is conducted under the online convex optimization framework, assuming boundedness of both the domain and function values. These assumptions are not adopted in our setting, and thus their results are not directly comparable to ours.
We summarize our key sample complexity results across various standard assumptions for the stochastic setting:
\begin{itemize}
    \item \textbf{Convex and $L_0$-Lipschitz continuous:} The one-point residual feedback ZO algorithm achieves a sample complexity of $\mathcal{O}(d^2 \sigma_0^2 \varepsilon^{-4})$, improving over the $\mathcal{O}(d^2 \sigma_0^8 \varepsilon^{-4})$ bound of~\citet{zhang2022new}. This improvement is significant in practice, as the variance $\sigma_0^2$ is often large, making higher-order variance terms (e.g., $\sigma_0^8$) extremely costly.
    
    \item \textbf{Convex and $L$-smooth:} Our analysis yields a sample complexity of $\mathcal{O}(d^2 \sigma_0^2 \varepsilon^{-3})$, maintaining the advantage over~\citet{zhang2022new} in terms of variance dependence.
    
    \item \textbf{Nonconvex and $L$-smooth:} We obtain a sample complexity of $\mathcal{O}(d^2 \varepsilon^{-6})$, improving upon the $\mathcal{O}(d^4 \varepsilon^{-6})$ result in~\citet{zhang2022new} by a factor of $d^2$.

    \item \textbf{Improved Bond for Least Squares Problem:} We provide Proposition \ref{prop:var} to justify that the gradient variance can be bounded much more tightly in terms of the data covariance structure for the least squares problem under mild conditions. 
    Consequently, the variance proxy $\sigma_0^2$ used in our general bounds can be replaced by a refined quantity that scales more favorably with the linear dimension.
    For example, this substitution directly implies sharper sample complexity results $\mathcal{O}(d\sigma_1^2\varepsilon^{-4})$ for convex and $L_0$-Lipschitz continuous objective functions (See all the refined results at the last row of Table \ref{table2}).

\end{itemize}

We summarize our contributions as follows:
\begin{itemize}
    \item[(i)] \textbf{Uniform ball smoothing for tighter gradient approximation.} 
We utilize the well-known uniform ball smoothing technique for constructing the smoothed objective function $f_\alpha(x)$ \citep{flaxman2004online,chen2022improve}, which yields a gradient approximation bound that is independent of the dimension $d$ (see Lemma~\ref{lem:sf}). 
In contrast, the Gaussian smoothing used by~\citet{zhang2022new} incurs a bound that grows rapidly with $d$. This dimension-independent property of uniform ball smoothing provides a significantly more accurate approximation to the true gradient in high-dimensional settings and serves as a key component enabling our improved convergence analysis.

\item[(ii)]  \textbf{Novel two-stage convergence analysis framework:} We propose a two-stage analysis framework for the convergence. 
We first establish convergence guarantees for the smoothed function $f_\alpha(x)$ (e.g., Lemma \ref{lem:dd}), fully leveraging its analytical properties—such as smoothness, gradient approximation bounds, and Lipschitz continuity. Then, using the tight bound from Lemma~\ref{lem:sf}, we transfer these results to the original function $f(x)$. In contrast,~\citet{zhang2022new} perform convergence analysis directly on $f(x)$ without decomposing through $f_\alpha(x)$, which amplifies estimation errors and leads to looser bounds with heavier dependence on the dimension $d$.

\item[(iii)] \textbf{Improved sample complexity}
Our use of uniform ball smoothing combined with a two-stage convergence analysis allows for more accurate gradient approximation and tighter control over dimension-dependent terms, resulting in significantly improved sample complexity across multiple problem settings (See Table~\ref{tb:det},~\ref{table2} and their discussions).
\end{itemize}

\section{Preliminaries}
\label{sec:prelim}

This section presents the necessary notations and assumptions that will be utilized in this paper, and the definition and property of the uniform ball smoothed objective function. Then the one-point residual feedback ZO algorithm is recalled.

\subsection{Notation and Assumption}

First, we introduce some important notions about objective functions.
\begin{definition}\label{def:CC}
A function $f(x)$ is $L_0$-Lipschitz continuous, if for any $x, y\in\RR^d$, it holds that 
\begin{equation}\label{eq:L0}
	|f(x) - f(y)| \leq L_0 \norm{x - y}, \quad\mbox{ for } \quad L_0 >0.
\end{equation}

\end{definition}

\begin{definition}
A function $f(x)$ is $L$-smooth, if for any $x, y\in\RR^d$, it holds that 
\begin{equation}\label{eq:L}
	\norm{\nabla f(x) - \nabla f(y)} \leq L \norm{x - y}, \quad\mbox{ for } \quad L >0.
\end{equation}
\end{definition}

\begin{definition}\label{ass:mu}
A function $f(x)$ is $\mu$-strongly convex with $\mu>0$, then  for any $x, y\in \RR^d$, it holds that
\begin{equation}\label{eq:mu}
	f(y) \ge f(x) + \dotprod{\partial f(x), y-x} + \frac{\mu}{2}\norm{y-x}^2,
\end{equation}
where $\partial f(x)$ is a subgradient of $f(x)$.
\end{definition}

In many practice, there exists randomness involving the problem to be solved.
Thus, the objective function is modeled as Eq.~\eqref{eq:prob_s}.
Next, we will introduce two assumptions related to the randomness of the objective function.

\begin{assumption}[Bounded Stochastic Value Variance]\label{ass:bnd_val_var}
We assume that for any $x\in\RR^d$, there exist $\sigma_0>0$ such that
\begin{equation}
	\EE_{\xi}\left[\Big(f(x,\;\xi) - f(x)\Big)^2\right] \leq \sigma_0^2.
\end{equation}

\end{assumption}

\begin{assumption}[Bound Stochastic Gradient Variance]
\label{ass:B_SG}
We assume that for any $x\in\RR^d$, there exist $\sigma_1>0$ such that
\begin{equation}
	\EE_{\xi}\left[\norm{\nabla f(x, \;\xi) - \nabla f(x)}^2\right] \leq \sigma_1^2
\end{equation}
\end{assumption}
Assumption~\ref{ass:bnd_val_var} is widely used in stochastic one-point feedback ZO optimization \citep{akhavan2020exploiting,zhang2022new}. 
Assumption~\ref{ass:B_SG} is almost a standard assumption for analyzing the convergence rate of stochastic gradient descent \citep{garrigos2023handbook,nemirovski2009robust}.

Note that the value variance $\sigma_0^2$ of function evaluation can be much smaller than the gradient variance $\sigma_1^2$ in many real applications.
We provide the following proposition to show its rationality.

\begin{proposition}\label{prop:var}
Given $\{(a_i, b_i)\}_{i=1}^{m}$ with $a_i\in\RR^d$ and $b_i \sim \cN(a_i^\top x^*, 1)$, we define the following least squares problem 
\begin{equation}\label{eq:lsr}
	\min_{x\in \mathbb{R}^d}f(x) = \frac{1}{m}\sum_{i=1}^{m} \left(a_i^\top x - b_i\right)^2 = \min_{x\in \mathbb{R}^d}\EE_{i\sim \mbox{Uniform}(1,m)} f_i(x), \mbox{ with } f_i(x) = \left(a_i^\top x - b_i\right)^2. 
\end{equation} 
If $\norm{a_i}^2\geq d$, then it holds that
\begin{equation}\label{eq:vv}
\EE\left[ \frac{1}{m}\sum_{i=1}^{m}(f_i(x) -  f(x))^2 \right]
\leq 
\frac{1}{d} \cdot \EE\left[ \frac{1}{m}\sum_{i=1}^{m}\norm{\nabla f_i(x) - \nabla f(x) }^2 \right].
\end{equation}
\end{proposition}

Proposition~\ref{prop:var} shows that for the least squares problem defined above, the function value variance $\sigma_0^2$ can be as small as $d^{-1}$ times the gradient variance $\sigma_1^2$. 
This observation plays an important role in our subsequent analysis of the stochastic setting. 
Since all our stochastic results are expressed in terms of $\sigma_0^2$, this structural property allows us, in this widely studied and practically important scenario, to replace $\sigma_0^2$ with $d^{-1}\sigma_1^2$. 
Consequently, the sample complexities can be further reduced, achieving improved dimension dependence. 
The last row of Table~\ref{table2} already summarizes all corresponding results, and thus we will not reiterate these implications in Section~\ref{sec:analysis_alg2}.

\begin{remark}
The condition $\|a_i\|^2 \ge d$ is natural in many statistical learning and deep model settings. 
First, many random design matrices satisfy the condition exactly; for instance, when using $\pm 1$ (Rademacher) entries in random projection matrices that are adopted in one-bit compressive sensing~\citep{li2018survey}, each row vector has squared norm exactly $d$ \citep{achlioptas2003database}.
Second, let us consider the sample vectors $a_i \in \mathbb{R}^d$ that are drawn from an isotropic sub-Gaussian design with independent coordinates and unit variance, then their Euclidean norm concentrates sharply around $\sqrt{d}$.  
In particular, by standard norm concentration results for sub-Gaussian vectors \cite[Theorem~3.1.1]{vershynin2018high}, one has
\[
\mathbb{P}\left( \left| \|a_i\| - \sqrt{d} \right| \geq t \right) \le 2\exp\!\left( - c t^2 / K^4 \right),
\]
which implies $\|a_i\|^2 = d(1 \pm o(1))$ with high probability.  
Third, by applying $\ell_2$ row-normalization and scaling to length $\sqrt{d}$, this condition can be enforced deterministically.  
Such normalization is widely used in practice: Batch Normalization \citep{ioffe2015batch} standardizes layer inputs to approximately zero mean and unit variance, encouraging feature isotropy, while $L_2$-normalization layers are standard in metric learning and face recognition \citep{schroff2015facenet}, often followed by scaling by $\sqrt{d}$.  
\end{remark}

\subsection{Subtle Properties of Smoothed Objective Function}

The objective of ZO optimization is to estimate the first-order gradient of a function using only function evaluations (zeroth-order oracles). 
To achieve this, it is necessary to perturb the function around the current point uniformly in all directions, thereby enabling accurate gradient estimation. 
This naturally motivates the smoothed formulation of $f$ introduced by~\citet{nesterov2017random}.
\citet{zhang2022new} follow the idea and utilize the Guaussian-smoothed $f$ to illustrate the nice convergence rate of the one-point residual feedback ZO algorithm.
This paper employs the smoothed objective with uniform distribution in a unit ball to refine the convergence analysis of the one-point residual feedback ZO algorithm.
We will show that the smoothed objective has several subtle properties that will be fully used in the analysis.

\begin{definition}
Function $f_\alpha(x)$ is a smooth version of $f(x)$ with 
\begin{equation}\label{eq:f_alp}
f_\alpha(x) = \EE_{\tu}\left[f(x+\alpha \tu)\right], \;\mbox{ with }\; \tu\sim\cB(0, I_d),
\end{equation}
where $\tu\sim \cB(0, I_d)$ means $\tu$ has uniform distribution in a unit ball.
\end{definition}

This smoothing scheme is often referred to as \emph{uniform ball smoothing}~\citep{hazan2016graduated,chen2017stochastic}, in contrast to the Gaussian smoothing~\citep{zhang2022new}.
Now we provide several important properties of $f_\alpha(x)$.
\begin{lemma}\label{lem:f_alp}
Letting $f_\alpha(x)$ be defined in Eq.~\eqref{eq:f_alp}, then it holds that $\EE[g_t] = \nabla f_{\alpha}(x_t)$ where $g_t$ is defined in Eq.~\eqref{eq:gt_def}.
Furthermore, $f_\alpha(x)$ also have the following properties:
\begin{enumerate}
	\item 	If $f(x)$ is $L_0$-Lipschitz continuous, then $f_\alpha(x)$ is also $L_0$-Lipschitz continuous. 
	\item  If $f(x)$ is $L$-smooth, then $f_\alpha(x)$ is also $L$-smooth. 
	\item If $f(x)$ is $\mu$-strongly convex, then $f_{\alpha}(x)$ is also $\mu$-strongly convex.
\end{enumerate} 
\end{lemma}

\begin{lemma}\label{lem:sf}
The smoothed function $f_\alpha(x)$ has the following properties:
\begin{enumerate}
	\item If $f(x)$ is convex, then it holds that $f_\alpha(x) \geq f(x)$.
	\item If $f(x)$ is $L_0$-Lipschitz continuous, it holds that $|f_\alpha(x) - f(x)| \leq L_0\alpha$.
	\item If $f(x)$ is $L$-smooth, it holds that
	\begin{align}\label{eq: approximate_falpha}
		|f_\alpha(x) - f(x)| \leq \frac{L\alpha^2}{2}, \; \mbox{ and }\; \norm{\nabla f_\alpha(x) - \nabla f(x)} \leq L\alpha. 
	\end{align}  
	Furthermore, if $f(x^*) \geq \infty$, then it holds that
	\begin{align*}
		f(x^*) \leq f(\tx^*) + \frac{L\alpha^2}{2},
	\end{align*}  
	where $x^*$ and $\tx^*$ are the minimal points of $f(x)$ and $f_\alpha(x)$, respectively.
\end{enumerate}
\end{lemma}

\begin{remark}
   We find that introducing the uniform ball smoothed objective function \eqref{eq:f_alp} can achieve a gradient approximation that satisfies \eqref{eq: approximate_falpha} in Lemma \ref{lem:sf}.
In contrast, \citet{zhang2022new} employ the Gaussian smoothed objective function:
$f_{\alpha,\text{Gaussian}}(x) = \mathbb{E}_{u \sim \mathcal{N}(0, I_d)}[f(x+\alpha u)],$
and Lemma 2 of~\citet{zhang2022new} shows
$$
\norm{\nabla f_{\alpha,\text{Gaussian}}(x) - \nabla f(x)} \leq L\alpha(d+3)^{3/2}.
$$
Thus, we indicate that uniform ball smoothing yields a significantly tighter approximation to the true gradient, particularly in high-dimensional settings. 
Because the approximation in \eqref{eq: approximate_falpha} does not depend on $d$.
This subtle property of the smoothed objective function has been studied by \citet{flaxman2004online,chen2022improve} and will be fully utilized in the following convergence analysis. 
\end{remark}

\subsection{ One-Point Residual Feedback ZO Algorithm }

Inspired by the two-point feedback ZO algorithms, \citet{zhang2022new} propose the one-point residual feedback ZO algorithm, which tries to exploit the information of the previous point to reduce the variance of the estimated gradient constructed by one-point feedback oracles. 
Given the (stochastic) gradient estimator by Eq.~\eqref{eq:gt_def} or Eq.~\eqref{eq:sgt_def}, the one-point residual feedback ZO algorithm will use it to update decision variables $x_t$.
The detailed algorithm description is listed in Algorithm~\ref{alg:CG} and Algorithm~\ref{alg:SCG}, for deterministic and stochastic optimization, respectively.

\begin{algorithm}[H]
\caption{One-Point Residual Feedback ZO Algorithm for Deterministic Optimization}
\label{alg:CG}
\begin{small}
	\begin{algorithmic}[1]
		\STATE {\bf Initialization:}
		Initial point $x_1$, step size $\eta$, iteration number $T$.
		\STATE Sample $u_0$ from the $d$-dimension unit spherical distribution;
		\FOR{$t=1,\dots,T$}
		\STATE Compute an approximate gradient $g_t$ by Eq.~\eqref{eq:gt_def};
		\STATE Update $x_{t+1} = x_t -\eta g_t$;
		\ENDFOR
		\STATE {\bf Return:}
		$x_T$.
	\end{algorithmic}
\end{small}
\end{algorithm}

\begin{algorithm}[H]
	\caption{One-Point Residual Feedback ZO Algorithm for Stochastic Optimization}
	\label{alg:SCG}
	\begin{small}
		\begin{algorithmic}[1]
			\STATE {\bf Initialization:}
			Initial point $x_1$, step size $\eta$, iteration number $T$.
			\FOR{$t=1,\dots,T$}
			\STATE Compute an approximate stochastic gradient $\tg_t$ by Eq.~\eqref{eq:sgt_def};
			\STATE Update $x_{t+1} = x_t -\eta \tg_t$;
			\ENDFOR
			\STATE {\bf Return:}
			$x_T$.
		\end{algorithmic}
	\end{small}
\end{algorithm}

The next section provides the two-stage convergence analysis framework of the above algorithms.


%
%
%
%
%
%

\section{Convergence Analysis for Algorithm \ref{alg:CG}}

\subsection{Smooth and Convex Case}\label{subsec:sm_cvx}

We will bound the variance of the estimated gradient $g_t$ defined in Eq.~\eqref{eq:gt_def} when the objective function is smooth.

\begin{lemma}\label{lem:gt_recs}
Suppose that $f(x)$ defined in Eq.~\eqref{eq:prob} is $L$-smooth. Letting $g_t$ be defined in Eq.~\eqref{eq:gt_def} and $\{x_t\}_{t=1}^T$ update as Algorithm~\ref{alg:CG} with $\eta \leq \frac{\alpha}{4dL_0}$, then the variance of $g_t$ satisfies that
\begin{equation}\label{eq:recs}
\EE\left[\norm{g_t}^2\right]
\leq
\frac{1}{2} \norm{g_{t-1}}^2
+ 8d\norm{\nabla f_\alpha(x_t)}^2
+ 8d \norm{\nabla f_\alpha(x_{t-1})}^2
+ 10d^2L^2\alpha^2.
\end{equation}
\end{lemma}

Next, we will provide the convergence analysis and sample complexity analysis of the one-point residual feedback ZO algorithm for smooth and convex functions.

\begin{lemma}\label{lem:dd}
    Suppose that $f(x)$ is $L_0$-Lipschitz, $L$-smooth and convex. 
Letting $g_t$ be defined in Eq.~\eqref{eq:gt_def} and $\{x_t\}_{t=1}^T$ update as Algorithm~\ref{alg:CG} with $\eta \leq \frac{1}{64dL}$, then it holds that
\begin{equation}\label{eq:scx_dd}
    \frac{1}{T}\sum_{t=1}^{T}\EE\Big[f_\alpha(x_t) - f_\alpha(x^*)\Big]
	\leq
	\frac{\norm{x_1 - x^*}^2}{\eta T}
	+ 20 \eta d^2L^2\alpha^2
	+ \frac{16\eta d L_0^2}{T}, 
\end{equation}
\end{lemma}

Combining Lemma \ref{lem:sf} and \ref{lem:dd}, the leading convergence theory of the one-point residual feedback ZO algorithm for smooth and convex functions is presented as follows.

\begin{theorem}\label{thm:main_aa}
Suppose that $f(x)$ is $L_0$-Lipschitz, $L$-smooth and convex. 
Letting $g_t$ be defined in Eq.~\eqref{eq:gt_def} and $\{x_t\}_{t=1}^T$ update as Algorithm~\ref{alg:CG} with $\eta = \min\left\{\frac{\alpha}{4dL_0}, \frac{1}{64dL}\right\}$, then it holds that
\begin{equation}\label{eq:main_ss}
\frac{1}{T}\sum_{t=1}^{T}\EE\Big[f(x_t) - f(x^*)\Big]
\leq 
\frac{\norm{x_1 - x^*}^2}{\eta T}
+ 20 \eta d^2L^2\alpha^2
+ \frac{16\eta d L_0^2}{T} 
+ \frac{L\alpha^2}{2}.
\end{equation}
\end{theorem}

\begin{remark}
    The subsequent theoretical results in this paper all follow a unified proof framework illustrated in this section. 
    The key idea is to first leverage the favorable properties of the smoothed function $f_\alpha(x)$ to bound the variance of the estimated gradient and establish convergence guarantees for $f_\alpha(x)$ (as shown in Lemma~\ref{lem:dd}). 
    In the second step, these convergence results are then transferred to the original objective function $f(x)$ via the tight approximation bound in Lemma~\ref{lem:sf}. This section provides a complete example of this two-stage analysis for the convex and $L$-smooth setting. 
    All subsequent proofs of other settings adhere to the same procedure, but for brevity, the intermediate results concerning $f_\alpha(x)$ are omitted from the main text and deferred to the Appendix.
\end{remark}

\begin{corollary}
\label{cor:main_ss}
Suppose that $f(x)$ satisfies the properties in Theorem~\ref{thm:main_aa}. 
Let $g_t$ be defined in Eq.~\eqref{eq:gt_def} and $\{x_t\}_{t=1}^T$ update as Algorithm~\ref{alg:CG} with $\eta = \frac{\alpha}{4dL_0}$ with $\alpha = \norm{x_1-x^*}^{2/3}\cdot\left(\frac{dL_0}{TL}\right)^{1/3}$.
If the iteration number satisfies $T \geq 2^{12} \norm{x_1-x^*}^2 d^4 \left(\frac{L}{L_0}\right)^2$,  Algorithm~\ref{alg:CG} has the following sample complexity with the target precision $\varepsilon$:
\begin{align*}
T = \max\left\{10^{3/2} \norm{x_1-x^*}^2 L L_0^{1/2} \cdot\frac{ d}{\varepsilon^{3/2}},\; \frac{8^{3/4}L_0 \norm{x_1-x^*}^{1/2}}{L^{1/4}}\cdot\frac{d^{1/4}}{\varepsilon^{3/4}}\right\}.
\end{align*}
\end{corollary}

Comparing the sample complexity $\cO\left(\frac{d}{\varepsilon^{3/2}}\right)$ with the one $\cO\left(\frac{d^{3/2}}{\varepsilon^{3/2}}\right)$ achieved by \citet{chen2022improve} and $\cO\left(\frac{d^3}{\varepsilon^{3/2}}\right)$ in the work \citep{zhang2022new}, our sample complexity has much lower dependency on the dimension $d$. 
Furthermore, our sample complexity does \emph{not} require the algorithm to take the momentum technique to achieve an improved convergence rate, which is needed in the work of \citet{chen2022improve}.

The sample complexity $\cO\left(\frac{d}{\varepsilon^{3/2}}\right)$ has lower dependence on $\varepsilon^{-1}$ than the one $\cO\left(\frac{d}{\varepsilon^2}\right)$ for the non-smooth convex functions shown in Eq.~\eqref{eq:ns_s}.
These advantages of sample complexity come from the smoothness of the objective function.
However, the sample complexity $\cO\left(\frac{d}{\varepsilon^{3/2}}\right)$ can not match the complexity $\cO\left(\frac{d}{\varepsilon}\right)$ achieved by two-point feedback algorithms \citep{nesterov2017random}.
This is because the step size $\eta$ depends on the value of $\alpha$, which is determined by the target precision.
In contrast, the step sizes of two-point feedback algorithms are not affected by the value $\alpha$ and the target precision.

Next, we will show that if the objective function is further $\mu$-strongly convex, Algorithm~\ref{alg:CG} can achieve faster convergence and a lower sample complexity. 

\begin{theorem}\label{thm:main}
Suppose that $f(x)$ is $L_0$ Lipschitz, $L$-smooth and $\mu$-strongly convex. 
Letting $g_t$ be defined in Eq.~\eqref{eq:gt_def} and $\{x_t\}_{t=1}^T$ update as Algorithm~\ref{alg:CG} with step size $\eta = \min\left\{\frac{\alpha}{4dL_0}, \frac{1}{56dL}\right\}$, then it holds that
\begin{equation}\label{eq:main}
\begin{aligned}
\EE\left[ f(x_{T+1}) - f(x^*)  \right]
\leq& 
\rho^t\Big(f(x_1) - f(x^*)\Big)
+ 12\rho^{t-1}\cdot dLL_0^2\eta^2  
+ \frac{15\eta d^2 L^3 \alpha^2  }{\mu}
+ 2L\alpha^2,
\end{aligned}
\end{equation}
with $\rho = 1 - \mu\eta$.
\end{theorem}

\begin{corollary}\label{cor:iter}
Suppose that $f(x)$ satisfies the properties in Theorem~\ref{thm:main}. 
Let $g_t$ be defined in Eq.~\eqref{eq:gt_def} and $\{x_t\}_{t=1}^T$ update as Algorithm~\ref{alg:CG} with $\eta = \frac{\alpha}{4dL_0}$ with $\alpha = \left(\frac{\varepsilon}{9L}\right)^{1/2}$ with the target precision $\varepsilon \leq \frac{9L_0^2\mu^2}{14^2 d^2 L^3}$.
Then,  Algorithm~\ref{alg:CG} has the following sample complexity:
\begin{equation}\label{eq:main_aa}
T = \frac{12dL_0 L^{1/2}}{\mu \varepsilon^{1/2}}\log\frac{3\Big(f(x_1) - f(x^*)\Big) }{\varepsilon}.
\end{equation}
\end{corollary}

Corollary~\ref{cor:iter} shows that Algorithm~\ref{alg:CG} can achieve a sample complexity $\cO\left(\frac{d}{\varepsilon^{1/2}}\log\frac{1}{\varepsilon}\right)$. 
This sample complexity is much lower than $\cO\left(\frac{d}{\varepsilon^{3/2}}\right)$, which is the sample complexity for smooth and convex functions.
However, compared with the two-point feedback method, which has a complexity $\cO\left(\frac{dL}{\mu}\log\frac{1}{\varepsilon}\right)$ \citep{nesterov2017random}, Algorithm~\ref{alg:CG} still has a larger sample complexity.
This is because the step size $\eta$ depends on the value of $\alpha$, which is determined by the target precision.

\subsection{Nonsmooth Convex Case}\label{subsec:nonsmooth}

First, we will bound the variance of the estimated gradient $g_t$ defined in Eq.~\eqref{eq:gt_def}.

\begin{lemma}\label{lem:g_var_det_non}
If the objective function $f(x)$ is $L_0$-Lipschitz continuous, and $g_t$ is constructed as Eq.~\eqref{eq:gt_def}, setting $\eta = \frac{\alpha}{3dL_0}$ in Algorithm~\ref{alg:CG}, then the variance of $g_t$ satisfies that
\begin{equation}\label{eq:g_dec}
\EE\left[\norm{g_t}^2\right] 
\leq
12 c_0 d L_0^2 .
\end{equation}
where $c_0$ is a non-negative numerical constant.
\end{lemma}

Next, we will provide the convergence analysis and sample complexity analysis of Algorithm \ref{alg:CG} for nonsmooth and convex functions.

\begin{theorem}\label{thm:main_ns}
Letting $f(x)$ defined in Eq.~\eqref{eq:prob} be convex and $L_0$-Lipschitz continuous and sequence $\{x_t\}_{t=1}^T$ be generated by Algorithm~\ref{alg:CG} with $\alpha = \sqrt{d/T}$ and $\eta = 1/(3\sqrt{dT}L_0)$, then it holds that
\begin{equation}\label{eq:ns}
\frac{1}{T}\sum_{t=1}^{T} \Big(f(x_t) - f(x^*)\Big)
\leq 
\frac{3\norm{x_1 - x^*}^2 d^{1/2} L_0}{T^{1/2}} 
+ \frac{4c_0 L_0d^{1/2} }{T^{1/2}} + \frac{d^{1/2}L_0}{T^{1/2}},
\end{equation}
where $x^*$ is the minimal point of $f(x)$.
\end{theorem}

\begin{corollary}
\label{cor:ns_s}
Let $f(x)$ satisfy the properties described in Theorem~\ref{thm:main_ns}. 
And parameters of Algorithm~\ref{alg:CG} are set as in Theorem~\ref{thm:main_ns}. 
To find an $\varepsilon$-suboptimal solution, then the sample complexity of Algorithm~\ref{alg:CG} is
\begin{equation}\label{eq:ns_s}
T = \max\left\{9^2\norm{x_1 - x^*}^2L_0^2,\; 12^2c_0^2 L_0^2,\; 9L_0^2\right\} \cdot \frac{d}{\varepsilon^2}.
\end{equation} 
\end{corollary}

Corollary~\ref{cor:ns_s} shows that Algorithm~\ref{alg:CG} can achieve the $\cO\left(\frac{d}{\varepsilon^2}\right)$ sample complexity for  non-smooth convex functions.
Our sample complexity is much lower than $\cO\left(\frac{d^2}{\varepsilon^2}\right)$ achieved by \citet{zhang2022new}.

Next, we will show that if the objective function is further $\mu$-strongly convex, Algorithm~\ref{alg:CG} can achieve faster convergence and a lower sample complexity. 
\begin{theorem}
\label{thm:main_ns_mu}
Let $f(x)$ be $\mu$-strongly convex and $L_0$-Lipschitz continuous and sequence $\{x_t\}_{t=1}^T$ be generated by Algorithm~\ref{alg:CG} with  $\eta = \alpha/(3dL_0)$.  Denoting $\rho = 1- \frac{\mu\eta}{2}$ and setting $w_t = \rho^{-t}$, then it holds that
\begin{equation}\label{eq:ns_ss}
\EE\left[f\left(\frac{\sum_{t=1}^{T}w_t x_t}{\sum_{t=1}^{T}w_t}\right) - f(x^*)\right] 
\leq
\frac{\mu \norm{x_1 - x^*}^2}{2\left(\rho^{-T} - 1\right)}
+ 
(4c_0 +1)L_0\alpha.
\end{equation}
\end{theorem}

\begin{corollary}
\label{cor:main_ns_mu}
Let $f(x)$ satisfy the properties described in Theorem~\ref{thm:main_ns}. 
And parameters of Algorithm~\ref{alg:CG} are set as in Theorem~\ref{thm:main_ns}. 
To find an $\varepsilon$-suboptimal solution, then the sample complexity of Algorithm~\ref{alg:CG} is
\begin{equation}\label{eq:ns_sss}
T = 6(4c_0 +1)L_0^2 \cdot \frac{d}{\mu \varepsilon} \cdot \log\left(\frac{\mu\norm{x_1 - x^*}^2}{\varepsilon} + 1\right).
\end{equation}
\end{corollary}

Corollary~\ref{cor:main_ns_mu} shows that our algorithm can achieve the $\cO\left(\frac{d}{\mu\varepsilon}\log\frac{1}{\varepsilon}\right)$ sample complexity when the objective function is non-smooth and $\mu$-strongly convex.
Compared with the complexity $\cO\left(\frac{d}{\varepsilon^2}\right)$ for non-smooth and convex functions,  Corollary~\ref{cor:main_ns_mu} shows that Algorithm~\ref{alg:CG} can achieve a much lower complexity when the objective function is strongly convex.

\subsection{Smooth and Nonconvex}
\label{subsec:snc}

\begin{theorem}
\label{thm:main_nc}
Suppose that $f(x)$ is $L_0$ Lipschitz, $L$-smooth but non-convex.  
Assume that $f(x^*) \geq -\infty$. 
Let $g_t$ be defined in Eq.~\eqref{eq:gt_def} and the sequence $\{x_t\}_{t=1}^T$ update as Algorithm~\ref{alg:CG} with step size $\eta = \min\left\{\frac{\alpha}{4dL_0}, \frac{1}{16dL} \right\} $, then it holds that
\begin{equation}\label{eq:main_nc}
\frac{1}{T}\sum_{t=1}^{T}\EE\left[\norm{\nabla f(x_t)}^2\right]
\leq
\frac{4\left(f(x_1) - f(x^*)\right)}{\eta T} 
+ \frac{4L\alpha^2}{\eta T}
+ \frac{32dLL_0^2 \eta}{T}
+ 40 \eta d^2 L^3 \alpha^2
+ 2L^2\alpha^2.
\end{equation}
\end{theorem}

\begin{corollary}
\label{cor:main_nc}
Suppose that $f(x)$ satisfies the properties in Theorem~\ref{thm:main_nc}.
Set the step size $\eta = \frac{\alpha}{4dL_0}$ with $\alpha = \left(\frac{dL_0 (f(x_1) - f(x^*))}{L^2T}\right)^{1/3}$.
If the iteration number satisfies 
\begin{align}\label{eq:T_cond}
T \geq \max\left\{ \frac{6^{3/2} dL_0}{(f(x_1) - f(x^*))^{1/2} L^{1/2}}, \frac{5^3\cdot d^4 (f(x_1) - f(x^*))}{L^2 L_0^2}, \frac{4^3 \cdot  L d^4 (f(x_1) - f(x^*))}{L_0^2}\right\},
\end{align}
then, to achieve $\frac{1}{T}\sum_{t=1}^{T}\EE\left[\norm{\nabla f(x_t)}^2\right] \leq \varepsilon^2$, the sample complexity of Algorithm~\ref{alg:CG} is:
\begin{align*}
T = \frac{24^{3/2} d L_0L (f(x_1) - f(x^*))}{\varepsilon^3}.
\end{align*}
\end{corollary}

Corollary~\ref{cor:main_nc} shows that to find a high precision solution, Algorithm~\ref{alg:CG} requires at most $\cO\left(\frac{d}{\varepsilon^3}\right)$ samples. 
Thus, our result has a great advantage over that of \citet{zhang2022new}, which is $\cO\left(\frac{d^3}{\varepsilon^3}\right)$.

\section{Convergence Analysis of Algorithm \ref{alg:SCG}}\label{sec:analysis_alg2}

\subsection{Nonsmooth and Convex Case}
\label{subsec:sns}

First, we will bound the variance of the estimated stochastic gradient $\tg_t$ defined in Eq.~\eqref{eq:sgt_def}.

\begin{lemma}\label{lem:sgt_dec}
Let the objective function $f(x)$ defined in Eq.~\eqref{eq:prob_s} be $L_0$-Lipschitz continuous and the stochastic gradient estimation $\tg_t$ be constructed as Eq.~\eqref{eq:sgt_def}. 
Suppose Assumption~\ref{ass:bnd_val_var} hold.
By setting the step size $\eta$ satisfy $\eta \leq \frac{\alpha}{3dL_0}$ in Algorithm~\ref{alg:SCG}, then the variance of $\tg_t$ satisfies that
\begin{equation}\label{eq:sg_dec}
	\EE\left[\norm{\tg_t}^2\right]
	\leq
	24 c_0 d L_0^2  + \frac{24d^2\sigma_0^2}{\alpha^2} .
\end{equation}
\end{lemma}

Next, we will provide the convergence analysis and sample complexity analysis of Algorithm \ref{alg:SCG} for nonsmooth and convex functions.
\begin{theorem}\label{thm:main_nss}
Letting $f(x)$ defined in Eq.~\eqref{eq:prob_s} be convex and $L_0$-Lipschitz continuous.  Suppose Assumption~\ref{ass:bnd_val_var} hold. 
Then sequence $\{x_t\}_{t=1}^T$ generated by Algorithm~\ref{alg:SCG} with  $\eta \leq \alpha/(3dL_0)$ satisfies the following property
\begin{equation}\label{eq:nss}
	\frac{1}{T}\sum_{t=1}^{T} \EE\Big[f(x_t) - f(x^*)\Big]
	\leq 
	\frac{\norm{x_1 - x^*}^2}{T\eta} + 24\eta c_0 dL_0^2 + \frac{24\eta d^2\sigma_0^2}{\alpha^2} + L_0\alpha,
\end{equation}
where $x^*$ is the minimal point of $f(x)$.
\end{theorem}

\begin{corollary}
\label{cor:main_ns}
Let $f(x)$ satisfy the properties described in Theorem~\ref{thm:main_nss}. 
Set step size  $\eta$ and parameter $\alpha$ of Algorithm~\ref{alg:SCG} as follows: 
\begin{equation}\label{eq:eta_alp}
\eta = \frac{\norm{x_1 - x^*} \alpha}{(24dT)^{1/2} (c_0 L_0^2\alpha^2 + d\sigma_0^2)^{1/2}}, 
\quad\mbox{with} \quad\alpha = \frac{(96)^{1/4}d^{1/2}\sigma_0^{1/2} \norm{x_1 - x^*}^{1/2}}{L_0^{1/2} T^{1/4}}.
\end{equation}
If the iteration number $T \ge \frac{3 L_0^2 \norm{x_1 - x^*}^2}{2^{11} \cdot c_0^2\sigma_0^2 }$, then to find an $\varepsilon$-suboptimal solution, the sample complexity of Algorithm~\ref{alg:SCG} is
\begin{equation}\label{eq:sns}
T = 96L_0^2\norm{x_1 - x^*}^2\cdot\max\left(\frac{16 d^2 \sigma_0^2}{\varepsilon^4},\frac{dc_0}{\varepsilon^2}\right).
\end{equation}
\end{corollary}

Eq.~\eqref{eq:sns} shows that when the target precision is small, then the sample complexity of Algorithm~\ref{alg:SCG} is $\cO\left(\max\left\{ \frac{d^2\sigma_0^2}{\varepsilon^4}, \frac{d c_0}{\varepsilon^2} \right\}\right)$ which has two independent parts. 
The first part $  \frac{d^2\sigma_0^2}{\varepsilon^4} $ comes from the variance of the stochastic function. 
The second part $\frac{d c_0}{\varepsilon^2}$ is almost the same as the one of Eq.~\eqref{eq:ns_s}, which comes from the variance of zeroth-order gradient estimation. 
We can observe that when $\varepsilon$ is sufficient small, then the term $\cO\left(\frac{d^2\sigma_0^2}{\varepsilon^4}\right)$ will dominate the sample complexity.
That is, the value variance of the stochastic function will dominate the convergence property of Algorithm~\ref{alg:SCG}.

Corollary~\ref{cor:main_ns} shows that to find a high precision solution, Algorithm~\ref{alg:SCG} requires $\cO\left(\frac{d^2\sigma_0^2}{\varepsilon^4}\right)$ samples which \emph{linearly} depends on the value variance $\sigma_0^2$. 
In contrast, the sample complexity of \citet{zhang2022new} is $\cO\left( \frac{d^2\sigma_0^8}{\varepsilon^4} \right)$.
Since the variance $\sigma_0^2$ is commonly a large value in practical applications, our result has a tighter bound over the dependence on the value of the variance.

Let us consider the $\mu$-strongly convex and $L_0$-Lipschitz continuous objective functions.
\begin{theorem} \label{thm:main_ns2}
Let $f(x)$ defined in Eq.~\eqref{eq:prob_s} be $\mu$-strongly convex and $L_0$-Lipschitz continuous. 
Suppose Assumption~\ref{ass:bnd_val_var} hold and 
sequence $\{x_t\}_{t=1}^T$ be generated by Algorithm~\ref{alg:SCG} with   $\eta \leq \alpha/(3dL_0)$.  Denoting $\rho = 1- \frac{\mu\eta}{2}$ and setting $w_t = \rho^{-t}$, then it holds that
\begin{equation}\label{eq:ns_s2}
\EE\left[f\left(\frac{\sum_{t=1}^{T}w_t x_t}{\sum_{t=1}^{T}w_t}\right) - f(x^*)\right] 
\leq 
\frac{\mu \norm{x_1 - x^*}^2}{2\left(\rho^{-T} - 1\right)}
+ 
\left(24c_0 d L_0^2 + \frac{24d^2\sigma_0^2}{\alpha^2}\right)\cdot\eta + L_0\alpha.
\end{equation}
\end{theorem}

\begin{corollary}\label{cor:main_ns2}
Let $f(x)$ satisfy the properties described in Theorem~\ref{thm:main_ns2}. 
And parameters of Algorithm~\ref{alg:SCG} are set as in Theorem~\ref{thm:main_ns2}. 
To find an $\varepsilon$-suboptimal solution with $\varepsilon <  4(8d(8c_0+1))^{1/2}\cdot \sigma_0$, then by setting step size $\eta = \frac{L_0\alpha^3}{24d^2\sigma_0^2}$ with $\alpha = \min\left\{\frac{\varepsilon}{8L_0},\left(\frac{d\sigma_0^2\varepsilon}{4c_0 L_0^3}\right)^{1/3}\right\}$, the sample complexity of Algorithm~\ref{alg:SCG} is
\begin{equation}\label{eq:sns2}
	T =  \frac{48 L_0^2}{\mu }\max\left(\frac{8^3 d^2 \sigma_0^2}{\varepsilon^3}, \frac{4c_0 d }{ \varepsilon}\right)\log\left(\frac{\mu\norm{x_1 - x^*}^2}{\varepsilon} + 1\right).
\end{equation}
\end{corollary}

\subsection{Smooth and Convex Case}
\label{subsec:scvx}

\begin{lemma}\label{lem:sgt_recs1}
Suppose that $f(x)$ defined in Eq.~\eqref{eq:prob_s} is $L_0$-Lipschitz continuous and $f(x,\;\xi)$ is $L$-smooth. 
Furthermore, we assume that Assumption~\ref{ass:bnd_val_var} and Assumption~\ref{ass:B_SG} hold.
Letting $\tg_t$ be defined in Eq.~\eqref{eq:sgt_def} and $\{x_t\}$ update as Algorithm~\ref{alg:SCG} with $\eta \leq \frac{\alpha}{8dL_0}$, then the variance of $\tg_t$ satisfies that
\begin{equation}\label{eq:srecs}
	\begin{aligned}
		\EE\left[\norm{\tg_t}^2\right] \leq& 
		\frac{1}{4} \norm{\tg_{t-1}}^2 
		+ 16d\norm{\nabla f_\alpha(x_t)}^2
		+ 16d \norm{\nabla f_\alpha(x_{t-1})}^2
		\\
		&
		+ \frac{64d^2\sigma_0^2}{\alpha^2} + 32d\sigma_1^2
		+ 10d^2L^2\alpha^2.
	\end{aligned}
\end{equation}
\end{lemma}

\begin{theorem}\label{thm:main_scvx}
Let $f(x)$ defined in Eq.~\eqref{eq:prob_s} be convex and $L_0$-Lipschitz continuous. Suppose function $f(x,\;\xi)$ is $L$-smooth and Assumption~\ref{ass:bnd_val_var}-\ref{ass:B_SG} hold. 
Letting sequence $\{x_t\}_{t=1}^T$ be generated by Algorithm~\ref{alg:SCG} with  $\eta \leq \min\left\{\frac{\alpha}{8dL_0}, \frac{1}{64 dL}\right\} $, then it holds that
\begin{equation}\label{eq:dec_snc}
\frac{1}{T}\sum_{t=1}^{T}\EE\Big[f(x_t) - f(x^*)\Big]
\leq 
\frac{\norm{x_1 - x^*}^2}{\eta T}
+ 20 \eta d^2L^2\alpha^2
+ \frac{16\eta d L_0^2}{T} 
+ \frac{L\alpha^2}{2}+ \frac{128d^2\sigma_0^2 \eta}{\alpha^2} + 64d\eta\sigma_1^2.
\end{equation}
\end{theorem}

\begin{corollary}\label{cor:main_scvx}
Let $f(x)$ satisfy the properties described in Theorem~\ref{thm:main_scvx}. 
Assume that the iteration number satisfies 
\begin{equation}\label{eq:T_rq_scvx}
T \geq \max\left\{ \frac{20^{3/2} L\norm{x_1 - x^*}^2}{ \sigma_0 d^4},\; \frac{L_0^2 \norm{x_1 - x^*}^2}{\sigma_0^2}, \; \frac{2^{7/2} \cdot L d^{1/2} \norm{x_1 - x^*}^2}{\sigma_0},\; \frac{L_0^2}{4\sigma_1^2}\right\}.
\end{equation}
Set $\alpha$ and step size $\eta$ as follows
\begin{equation}\label{eq:alp_eta_scvx}
\alpha = 4\left(\frac{d \norm{x_1 - x^*}\sigma_0}{L}\right)^{1/3}\cdot \frac{1}{T^{1/6}},\quad\mbox{ and }\quad \eta = \frac{\alpha \norm{x_1 - x^*}}{8\sqrt{T} d\sigma_0}.
\end{equation}
To find an $\varepsilon$-suboptimal solution, then the sample complexity of Algorithm~\ref{alg:SCG} is
\begin{equation}\label{eq:T_cplx}
T = \max\left\{ \frac{12^3 \cdot L d^2 \norm{x_1 - x^*}^2 \sigma_0^2}{\varepsilon^3},\; \frac{2^{21/2}\cdot d^{1/2} \norm{x_1 - x^*}^2 \sigma_1^3 }{L^{1/2} \sigma_0 \varepsilon^{3/2}}\right\}.
\end{equation}
\end{corollary}

Eq.~\eqref{eq:T_cplx} shows that Algorithm~\ref{alg:SCG} has a complexity $T = \cO\left( \max\left\{\frac{d^2\sigma_0^2}{\varepsilon^3}, \frac{d^{1/2} \sigma_1^3}{\sigma_0 \varepsilon^{3/2}} \right\}\right)$.
Thus, when the target precision $\varepsilon$ is sufficiently small, then the sample complexity will reduce to $T = \cO\left(\frac{d^2\sigma_0^2}{\varepsilon^3}\right)$ which heavily depends on the dimension. 
In many practical applications, it holds that $\sigma_1^2 = c_\sigma\cdot \sigma_0^2$ with $c_\sigma = \cO(d)$ (refer to Proposition~\ref{prop:var}).  
In these applications,   Algorithm~\ref{alg:SCG} can achieve a sample complexity $T = \cO\left(\frac{d\sigma_1^2}{\varepsilon^3}\right)$ which linearly depends on the dimension and standard variance assumption on the stochastic gradient. 

Eq.~\eqref{eq:T_cplx} shows that to find a high precision solution, Algorithm~\ref{alg:SCG} requires $\cO\left(\frac{d^2\sigma_0^2}{\varepsilon^3}\right)$ samples which \emph{linearly} depends on the value variance $\sigma_0^2$. 
In contrast, the sample complexity of \citet{zhang2022new} is $\cO\left( \frac{d^2\sigma_0^6}{\varepsilon^4} \right)$ which \emph{cubically} depends on the variance.
Thus, our result has a tighter bound over the dependence on the value variance.

\begin{theorem}\label{thm:main_sncvx}
Let $f(x)$ defined in Eq.~\eqref{eq:prob_s} be $\mu$-strongly convex and $L_0$-Lipschitz continuous. Suppose function $f(x,\;\xi)$ is $L$-smooth and Assumption~\ref{ass:bnd_val_var}-\ref{ass:B_SG} hold.
Letting $\tg_t$ be defined in Eq.~\eqref{eq:sgt_def} and $\{x_t\}_{t=1}^T$ update as Algorithm~\ref{alg:SCG} with step size $\eta \leq \min\left\{\frac{\alpha}{8dL_0}, \frac{1}{112dL}\right\}$, then it holds that
\begin{equation}\label{eq:main_dec}
\begin{aligned}
	\EE\left[ f(x_{T+1}) - f(x^*)  \right]
	\leq& 
	\rho^t\Big(f(x_1) - f(x^*) \Big)
	+ 24\rho^{t-1}\cdot dLL_0^2\eta^2  
	\\
	&+ \frac{3L\eta}{\mu}\left(\frac{32d^2\sigma_0^2}{\alpha^2} + 16d\sigma_1^2
	+ 5d^2L^2\alpha^2\right)
	+ 2L\alpha^2,
\end{aligned}
\end{equation}
with $\rho = 1 - \mu\eta$.
\end{theorem}

\begin{corollary}\label{cor:main_sncvx}
Let $f(x)$ satisfy the properties described in Theorem~\ref{thm:main_sncvx}. 
If the target precision $\varepsilon\le \min\left\{ 64\sigma_0,\; \frac{16\sigma_1^2}{d L} \right\}$,  setting the step size $\eta = \frac{\mu \alpha^4}{48d^2\sigma_0^2} $ with $\alpha^2 \leq \min\left\{\frac{\varepsilon}{32L},\;\left(\frac{d\sigma_0^2 \varepsilon}{4L\sigma_1^2}\right)^{1/2}\right\}$, then the sample complexity of Algorithm~\ref{alg:SCG} is
\begin{equation}
T = \max\left\{\frac{3\cdot2^{14}\cdot  L^2d^2\sigma_0^2}{\mu^2\varepsilon^2},\; \frac{192dL\sigma_1^2}{\mu^2\varepsilon}\right\}   \log\frac{3\Big(f(x_1) - f(x^*) \Big)}{\varepsilon}. 
\end{equation}
\end{corollary}

\subsection{Smooth and Nonconvex}\label{subsec:sncvx}
\begin{theorem}\label{thm:main_snc}
Let $f(x)$ defined in Eq.~\eqref{eq:prob_s} be $L_0$-Lipschitz continuous. Suppose function $f(x,\;\xi)$ is also $L$-smooth and Assumption~\ref{ass:bnd_val_var}-\ref{ass:B_SG} hold. 
Letting sequence $\{x_t\}_{t=1}^T$ be generated by Algorithm~\ref{alg:SCG} with  $\eta \leq \min\left\{\frac{\alpha}{8dL_0}, \frac{1}{32 dL}\right\} $, then it holds that
\begin{small}
\begin{equation}
\begin{aligned}
\frac{1}{T}\sum_{t=1}^{T}\EE\left[\norm{\nabla f(x_t)}^2\right]
\leq
\frac{4\left(f(x_1) - f(x^*)\right)}{\eta T} 
+ \frac{4L\alpha^2}{\eta T}
+ \frac{32dLL_0^2 \eta}{T}
+ 40 \eta d^2 L^3 \alpha^2
+ 2L^2\alpha^2	+  \frac{128d^2\sigma_0^2\eta }{\alpha^2}
+ 64d\sigma_1^2\eta.
\end{aligned}
\end{equation}
\end{small}

\end{theorem}

\begin{corollary}\label{cor:main_snc}
Let $f(x)$ satisfy the properties described in Theorem~\ref{thm:main_snc}. 
Set the step size $\eta$ and parameter $\alpha$ as follow:
\begin{equation}\label{eq:eta_alp1}
\eta = \frac{\alpha (f(x_1) - f(x^*))^{1/2} }{4\sqrt{T}d\sigma_0} \quad \mbox{and} \quad \alpha=	\frac{(f(x_1) - f(x^*))^{1/6} (d\sigma_0)^{1/3}}{L^{2/3}T^{1/6}}.
\end{equation} 
If the iteration number satisfies $T \geq \frac{(f(x_1) - f(x^*)) L_0^2}{d\sigma_0^2}$, then the sample complexity of Algorithm~\ref{alg:SCG} is
\begin{equation}
T = (f(x_1) - f(x^*)) \cdot \max\left\{ \frac{140^3 \cdot  d^2 L^2 \sigma_0^2}{\varepsilon^6}, \frac{40\cdot Ld^2}{\varepsilon^2}, \frac{2^9 \cdot d^{1/2} \sigma_1^3}{L\sigma_0 \varepsilon^3}, \frac{96^{3/2} \cdot d^2\sigma_0^2 L^{1/2}}{\varepsilon^3 (f(x_1) - f(x^*))^{3/2}}\right\}. 
\end{equation}
\end{corollary}

Our sample complexity is $\cO\left(\frac{d^2}{\varepsilon^6}\right)$ which is much lower than $\cO\left(\frac{d^4}{\varepsilon^6}\right)$ achieved by \citet{zhang2022new}.

\section{Conclusion}

This paper considers a long-standing open question in zeroth-order optimization by showing that one-point feedback algorithms can achieve the same optimal sample complexity as two-point methods in several fundamental settings. 
By employing uniform ball smoothing and a refined two-stage convergence analysis---first establishing rates for the smoothed objective and then transferring them to the original function---we prove that the one-point residual feedback ZO algorithm of~\citet{zhang2022new} attains $\mathcal{O}(d\varepsilon^{-2})$ complexity for convex Lipschitz objectives, matching the optimal two-point bound. 
We further establish improved complexities for convex smooth and nonconvex smooth objectives, as well as sharper bounds in the stochastic setting with significantly reduced dependence on the value variance. 
Nevertheless, in the stochastic regime, our bounds do not yet fully achieve linear dimension dependence in complete generality, and future research may explore broader conditions or richer problem classes under which such optimal rates can be attained.
Our results demonstrate that optimality in one-point feedback ZO optimization can be achieved without modifying the underlying algorithm.

\pb
\bibliography{ref.bib}
\bibliographystyle{apalike2}

\appendix

\section{Some Useful Lemmas}

\begin{lemma}\label{lem:uua}
Letting $u\sim \cS(0, I_d)$, that is, $u$ uniformly samples from the sphere of a unit ball, and $a \in \RR^d$ be a constant vector, then it holds that
\begin{equation}\label{eq:uuas}
\EE\left[\norm{uu^\top a}^2\right] =  \frac{\norm{a}^2}{d}.
\end{equation}
\end{lemma}
\begin{proof}
It holds that
\begin{align*}
\EE\left[\norm{uu^\top a}^2\right] = \EE\left[a^\top uu^\top u u^\top a\right] = \EE\left[a^\top uu^\top a\right] = \frac{\norm{a}^2}{d}.
\end{align*}
\end{proof}

\begin{lemma}\label{lem:Delta}
Letting  non-negative sequences $\{\Delta_t\}$ and $\{M_t\}$ with $t=1,\dots, T$ satisfies $\Delta_{t+1} \leq \rho \Delta_t + M_t$ with $0<\rho <1$, then it holds that
\begin{equation}\label{eq:Delta}
\Delta_{t+1} \leq \rho^t \Delta_1 + \sum_{i=0}^{t-1} \rho^i M_{t-i}
\end{equation} 
\end{lemma}
\begin{proof}
We have
\begin{align*}
\Delta_{t+1} \leq& \rho \Delta_t + M_t \leq \rho ( \rho \Delta_{t-1} + M_{t-1}) + M_t = \rho^2 \Delta_{t-1} + \rho M_{t-1} + M_t.
\end{align*}
Using above equation recursively, we can obtain that
\begin{align*}
\Delta_{t+1} \leq \rho^t \Delta_1 + \sum_{i=0}^{t-1} \rho^i M_{t-i}.
\end{align*}
\end{proof}

\begin{lemma}
	\label{lem:ss}
Letting  non-negative $\{M_t\}$ with $t=1,\dots, T$ and a constant $0<\rho <1$, then it holds that
\begin{equation*}
\sum_{t=1}^{T} \sum_{i=0}^{t-1} \rho^i M_{t-i}  
\leq 
\frac{1}{1-\rho} \sum_{t=1}^T M_t .
\end{equation*} 
\end{lemma}
\begin{proof} We have
\begin{align*}
\sum_{t=1}^{T} \sum_{i=0}^{t-1} \rho^i M_{t-i} 
= 
\sum_{t=1}^{T} \sum_{j= t}^1 \rho^{t - j} M_j 
= 
\sum_{t=1}^{T} \sum_{j= 1}^t \rho^{t - j} M_j
\leq \frac{1}{1-\rho} \sum_{t=1}^T M_t. 
\end{align*}
\end{proof}

\begin{lemma}[Lemma 9 of \citet{shamir2017optimal}]
\label{lem:g_var}
For any function $g$ which is $L_0$-Lipschitz with respect to the Frobenius norm, it holds that
if $u$ is uniformly distributed on the Euclidean unit sphere, then
\begin{align*}
\sqrt{\EE\left[\left(g(u) - \EE[g(u)]\right)^4\right]} 
\leq
\frac{c_0 L_0^2}{d}, 
\end{align*}
for some numerical constant $c_0> 0$.
\end{lemma}

\begin{lemma}[Theorem~2.1.10 of \citet{nesterov2013introductory}]\label{lem:mu}
If $f(x)$ is smooth and $\mu$-strongly convex, then it holds that
\begin{equation*}
2\mu(f(x) - f(x^*)) \leq \norm{\nabla f(x)}^2.
\end{equation*}
\end{lemma}

\section{Proofs of Section~\ref{sec:prelim}}

\begin{proof}[Proof of Lemma~\ref{lem:f_alp}]
	The result that $\EE[g_t] = \nabla f_{\alpha}(x_t)$ is standard in zeroth-order optimization \citep{flaxman2004online}.
	
	If $f(x)$ is $L_0$-Lipschitz continuous, then we have
	\begin{align*}
		|f_\alpha(x) - f_\alpha(y)| 
		=
		|\EE_u \left[f(x+\alpha u) - f(y+\alpha u)\right] | 
		\leq
		\EE\left[| f(x+\alpha u) - f(y+\alpha u) |\right] 
		\leq
		L_0\norm{x - y},
	\end{align*}
	where the first inequality is because of Jensen's inequality and the last inequality is because  $f(x)$ is $L_0$-Lipschitz continuous.
	
	If $f(x)$ is $L$-smooth, then we have
	\begin{align*}
		&\norm{\nabla f_\alpha(x) -\nabla f_\alpha(y)} 
		= \norm{\EE_u\left[ \nabla f(x+\alpha u) - \nabla f(y+\alpha u) \right]}\\
		\leq&  
		\EE_u\left[\norm{ \nabla f(x+\alpha u) - \nabla f(y+\alpha u) }\right]
		\leq L\norm{x - y}.
	\end{align*}
	Thus, $f_\alpha(x)$ is also $L$-smooth.

	Furthermore,
	\begin{align*}
		&f_\alpha(y) = \EE_u\left[f(y + \alpha u)\right]\\ 
		\geq& 
		\EE_u\left[f(x+\alpha u) + \dotprod{\nabla f(x+\alpha u), y+\alpha u - (x+\alpha u)} + \frac{\mu}{2} \norm{(y+\alpha u) - (x+\alpha u)}^2\right]\\
		=& 
		f_\alpha(x) + \dotprod{\nabla f_\alpha(x), y - x} +\frac{\mu}{2}\norm{y-x}^2,
	\end{align*}
	where the first inequality is because $f(x)$ is $\mu$-strongly convex.
	Thus,  $f_{\alpha}(x)$ is also $\mu$-strongly convex.
	
\end{proof}

\begin{proof}[Proof of Lemma~\ref{lem:sf}]
	If $f(x)$ is convex, then it holds that 
	\begin{align*}
		f_\alpha(x) =\EE_{\tu}	\left[f(x + \alpha \tu)\right] \geq f(x) + \EE_{\tu}\left[\dotprod{\nabla f(x), \alpha \tu}\right] = f(x).
	\end{align*}
	
	If $f(x)$ is $L_0$-Lipschitz continuous, then it holds that
	\begin{align*}
		|f_\alpha(x) - f(x)|
		=
		|\EE_{\tu}\left[f(x+\alpha \tu)\right] - f(x)|
		\leq
		L_0 \alpha\EE_{\tu}\left[\norm{\tu}\right] \leq L_0\alpha.
	\end{align*}
	If $f(x)$ is $L$-smooth, then it holds that
	\begin{align*}
		\norm{\nabla f_\alpha(x) - \nabla f(x)}
		=
		\norm{\EE_{\tu}\left[\nabla f(x+\alpha\tu )- \nabla f(x)\right]}
		\leq
		\EE_{\tu}\left[\norm{\nabla f(x+\alpha \tu) - \nabla f(x)}\right]
		\leq
		L\alpha.
	\end{align*}
	If $f(x^*) \ge -\infty$ and $\tx^*$ is the minimal point of $f_\alpha(x)$, then 
	\begin{align*}
		f(x^*) \leq f(\tx^*) \leq f_\alpha(\tx^*) + \frac{L\alpha^2}{2}.
	\end{align*}
\end{proof}

\begin{proof}[Proof of Proposition~\ref{prop:var}]
	The value variance related to Eq.~\eqref{eq:lsr} can be bounded as 
	\begin{equation}\label{eq:sig0}
		\begin{aligned}
			&\EE\left[ \frac{1}{m}\sum_{i=1}^{m}(f_i(x) -  f(x))^2 \right]
			= \EE\left[\frac{1}{m}\sum_{i=1}^{m} (f_i(x))^2 - (f(x))^2\right]
			\leq 
			\frac{1}{m}\sum_{i=1}^{m}\EE\left[ (a_i^\top x - b_i)^2 \right] \\
			=& 
			\frac{1}{m}\sum_{i=1}^{m}\EE\left[ (a_i^\top x - a_i^\top x^* + a_i^\top x^* - b_i)^2 \right] 
			= \frac{1}{m}\sum_{i=1}^{m} \left[ \left(a_i^\top (x - x^*)\right)^2 + 1 \right] \\
			=& \frac{1}{m}\sum_{i=1}^{m} \left(a_i^\top (x - x^*)\right)^2 + 1,
		\end{aligned}
	\end{equation}
	where the third inequality is because of the assumption $ (a_i^\top x^* - b_i) \sim \cN(0,1)$.

	We also have $\nabla f_i(x) = 2 a_i (a_i^\top x - b_i)$.
	Thus, we can obtain that
	\begin{align*}
		\EE\left[\norm{\nabla f_i(x)}^2\right] = 4 \EE\left[\norm{ a_i (a_i^\top x - a_i^\top x^* + a_i^\top x^* - b_i ) }^2\right]
		=4\norm{a_i}^2  \left(\left(a_i^\top( x - x^*)\right)^2  + 1\right).
	\end{align*}
	
	Thus,
	\begin{equation}\label{eq:sig1}
		\begin{aligned}
			\EE\left[ \frac{1}{m}\sum_{i=1}^{m}\norm{\nabla f_i(x) - \nabla f(x) }^2 \right]
			\leq & 
			\frac{1}{m} \left(1 - \frac{1}{m}\right)\sum_{i=1}^{m} \EE\left[\norm{\nabla f_i(x)}^2\right]\\
			=&
			\frac{4}{m}\left(1 - \frac{1}{m}\right)\sum_{i=1}^{m}\norm{a_i}^2  \left(\left(a_i^\top( x - x^*)\right)^2  + 1\right).
		\end{aligned}
	\end{equation}
	
	If $\norm{a_i}^2 \geq d$, comparing Eq.~\eqref{eq:sig0} and Eq.~\eqref{eq:sig1}, we can obtain that  
	\begin{equation*}
		\EE\left[ \frac{1}{m}\sum_{i=1}^{m}(f_i(x) -  f(x))^2 \right]
		\leq 
		\frac{1}{d} \cdot \EE\left[ \frac{1}{m}\sum_{i=1}^{m}\norm{\nabla f_i(x) - \nabla f(x) }^2 \right].
	\end{equation*}
	
\end{proof}

\section{Proofs of Section~\ref{subsec:sm_cvx}}

\begin{proof}[Proof of Lemma~\ref{lem:gt_recs}]
First, we represent Eq.~\eqref{eq:gt_def} as follows:
\begin{equation}\label{eq:gts}
	\begin{aligned}
		g_t 
		=& d\cdot \frac{f_\alpha(x_t+\alpha u_t) - f_\alpha(x_{t-1} + \alpha u_{t-1})}{\alpha} u_t 
		+ d\cdot\frac{f(x_t + \alpha u_t) - f_\alpha(x_t+\alpha u_t)}{\alpha}u_t\\
		+& d\cdot\frac{ f_\alpha(x_{t-1}+\alpha u_{t-1}) - f(x_{t-1} + \alpha u_{t-1}) }{\alpha}u_t
	\end{aligned}
\end{equation}

For the notation convenience, we denote that
\begin{align*}
	g_t' = d\cdot \frac{f_\alpha(x_t+\alpha u_t) - f_\alpha(x_{t-1} + \alpha u_{t-1})}{\alpha} u_t.
\end{align*}
By the Taylor's expansion, we have 
\begin{align*}
	f_\alpha(x_t + \alpha u_t) =& f_\alpha(x_t) + \alpha \dotprod{\nabla_\alpha f(x_t), u_t} + \phi(x_t, u_t,\alpha),\\
	f_\alpha(x_{t-1} + \alpha u_{t-1}) =& f_\alpha(x_{t-1}) + \alpha \dotprod{\nabla f_\alpha(x_{t-1}), u_{t-1}} + \phi(x_{t-1}, u_{t-1},\alpha),
\end{align*} 
where $\phi(x_t, u_t, \alpha) = f_\alpha(x_t + \alpha u_t) - \Big(f_\alpha(x_t) + \alpha \dotprod{\nabla_\alpha f(x_t), u_t}\Big) $.
Thus, we can obtain that
\begin{align*}
	g_t' 
	&= d\cdot \left(u_tu_t^\top\nabla f_\alpha(x_t) - u_tu_{t-1}^\top\nabla f_\alpha(x_{t-1})\right)\\
	& 
	+ d\cdot \frac{ f_\alpha(x_t) - f_\alpha(x_{t-1})}{\alpha}u_t 
	+ d\cdot \frac{\phi(x_t, u_t,\alpha) - \phi(x_{t-1}, u_{t-1},\alpha)}{\alpha}u_t.
\end{align*}
Accordingly, we have
\begin{equation}\label{eq:four}
	\begin{aligned}
		&\EE\left[\norm{g_t'}^2 \right]
		\leq
		4d^2 \EE\left[ \norm{u_tu_t^\top \nabla f_\alpha(x_t)}^2 \right] + 4d^2 \EE\left[  \Big( u_{t-1}^\top \nabla f_\alpha(x_{t-1}) \Big)^2\right] \cdot \EE\left[ \norm{u_t}^2 \right] \\
		&+ \frac{4d^2(f_\alpha(x_t) - f_\alpha(x_{t-1}))^2}{\alpha^2} \EE\left[\norm{u_t}^2\right] 
		+ \EE\left[\frac{4d^2\Big(\phi(x_t, u_t,\alpha) - \phi(x_{t-1}, u_{t-1},\alpha)  \Big)^2}{\alpha^2} \norm{u_t}^2 \right].
	\end{aligned}
\end{equation}

By the $L_0$-Lipschitz continuous of $f(x)$ and Lemma~\ref{lem:f_alp}, we can obtain that
\begin{align*}
	\Big(f_\alpha(x_t) - f_\alpha(x_{t-1})\Big)^2 
	\leq
	L_0^2\norm{x_t - x_{t-1}}^2 = L_0^2\eta^2 \norm{g_{t-1}}^2. 
\end{align*}

By Eq.~\eqref{eq:uuas}, we can obtain that
\begin{equation*}
	\EE_{u_t}\left[\norm{u_tu_t^\top \nabla f_\alpha(x_t)}^2\right] = \frac{1}{d}\norm{\nabla f_\alpha(x_t)}^2, \mbox{ and } \EE_{u_t}\left[\norm{u_t}^2\right] = 1.
\end{equation*} 
Furthermore, it holds that 
\begin{align*}
	\EE_{u_{t-1}}\left[\Big(u_{t-1}^\top \nabla f_\alpha(x_{t-1})\Big)^2\right] = \frac{1}{d}\norm{\nabla f_\alpha(x_{t-1})}^2.
\end{align*}

Since $f(x)$ is $L$-smooth, we can conclude that $f_\alpha(x)$ is also $L$-smooth by Lemma~\ref{lem:f_alp}.
Thus, by Eq.~\eqref{eq:L}, we can obtain that
\begin{equation*}
	|\phi(x_t, u_t,\alpha)| \leq \frac{L\alpha^2}{2}\norm{u_t}^2 = \frac{L\alpha^2}{2}.
\end{equation*}
Accordingly, we have
\begin{align*}
	\EE_{u_t}\left[\Big(\phi(x_t, u_t,\alpha) - \phi(x_{t-1}, u_{t-1},\alpha)  \Big)^2 \norm{u_t}^2  \right]
	\leq
	\frac{L^2\alpha^4}{2}
	+
	\frac{L^2\alpha^4}{2}
	=
	L^2\alpha^4.
\end{align*}
Combining above results with Eq.~\eqref{eq:four} and using full expectation rule, we can obtain that
\begin{align*}
	\EE\left[\norm{g_t'}^2\right]
	=&	\EE_{u_{t-1}}\left[\EE_{u_t}\left[\norm{g'_t}^2 \mid u_{t-1} \right]\right]\\
	\leq& 
	4d\norm{\nabla f_\alpha(x_t)}^2
	+ 4d^2\EE_{u_{t-1}}\left[\Big( u_{t-1}^\top \nabla f_\alpha(x_{t-1}) \Big)^2\right]
	+ \frac{4d^2L_0^2\eta^2}{\alpha^2} \norm{g_{t-1}}^2
	+ 4d^2L^2\alpha^2\\
	=&
	4d \norm{\nabla f_\alpha(x_t)}^2 
	+ 4d \norm{\nabla f_\alpha(x_{t-1})}^2
	+ \frac{4d^2L_0^2\eta^2}{\alpha^2} \norm{g_{t-1}}^2
	+ 4d^2L^2\alpha^2.
\end{align*}
Combining with Eq.~\eqref{eq:gts}, we can obtain that
\begin{align*}
	\EE\left[\norm{g_t}^2\right]
	\leq&
	2\EE\left[\norm{g_t'}^2\right]
	+ 4d^2 \EE\left[\frac{(f_\alpha(x_t) - f(x_t))^2 + (f_\alpha(x_{t-1}) - f(x_{t-1}))^2}{\alpha^2}\norm{u_t}^2  \right]\\
	\leq&
	2\EE\left[\norm{g_t'}^2\right]
	+ 2d^2L^2\alpha^2\\
	\leq&
	8d\norm{\nabla f_\alpha(x_t)}^2
	+ 8d \norm{\nabla f_\alpha(x_{t-1})}^2
	+ \frac{8d^2L_0^2\eta^2}{\alpha^2} \norm{g_{t-1}}^2
	+ 10d^2L^2\alpha^2\\
	\leq&
	8d\norm{\nabla f_\alpha(x_t)}^2
	+ 8d \norm{\nabla f_\alpha(x_{t-1})}^2
	+ \frac{1}{2} \norm{g_{t-1}}^2
	+ 10d^2L^2\alpha^2,
\end{align*}
where the second inequality is because of $|f(x)-f_\alpha(x)| \leq \frac{L\alpha^2}{2}$ by Lemma~\ref{lem:sf} and the last inequality is because of $\eta \leq \frac{\alpha}{4dL_0} $.
\end{proof}

\begin{proof}[Proof of Lemma~\ref{lem:dd}]
    By the update rule of Algorithm~\ref{alg:CG}, we can obtain that
\begin{equation}\label{eq:xx}
	\begin{aligned}
		\EE\left[\norm{x_{t+1} - x^*}^2\right]
		=&
		\norm{x_t - x^*}^2 - \eta\dotprod{\nabla f_\alpha(x_t), x_t - x^*} + \eta^2 \EE\left[\norm{g_t}^2\right]\\
		\leq&
		\norm{x_t - x^*}^2 - \eta \Big(f_\alpha(x_t) - f_\alpha(x^*)\Big) - \frac{\eta}{2L}\norm{\nabla f_\alpha(x_t)}^2 + \eta^2\EE\left[\norm{g_t}^2\right],
	\end{aligned}
\end{equation}
where the last inequality is because of the $L$-smoothness of $f_\alpha(x)$ implied by Lemma~\ref{lem:sf}.

Representing  above equation and summing it from $t=1$ to $T$, we can obtain that
\begin{align*}
	&\eta \sum_{t=1}^{T} \EE \Big[f_\alpha(x_t) - f_\alpha(x^*)\Big] \\
	\leq& 
	\sum_{t=1}^{T} \left[\EE\left[\norm{x_t - x^*}^2\right] - \EE\left[\norm{x_{t+1} - x^*}^2\right]\right]
	-
	\frac{\eta}{2L}\sum_{t=1}^{T} \norm{\nabla f_\alpha(x_t)}^2
	+
	\eta^2\sum_{t=1}^{T}\EE\left[\norm{g_t}^2\right]\\
	=& 
	\norm{x_1 - x^*}^2 - \EE\left[\norm{x_{T+1} - x^*}^2\right] 
	-
	\frac{\eta}{2L}\sum_{t=1}^{T} \norm{\nabla f_\alpha(x_t)}^2
	+
	\eta^2\sum_{t=1}^{T}\EE\left[\norm{g_t}^2\right].
\end{align*}

By Eq.~\eqref{eq:recs} and using Lemma~\ref{lem:Delta} and Lemma~\ref{lem:ss} with $\rho = 1/2$ and $M_t = 8d(\norm{\nabla f_\alpha(x_t)}^2 + \norm{\nabla f_\alpha(x_{t-1})}^2) + 10d^2L^2\alpha^2$, we can obtain that
\begin{equation}\label{eq:gt_sum}
	\begin{aligned}
		\sum_{t=1}^{T}	\EE\left[\norm{g_t}^2\right] 
		\leq&
		\sum_{t=1}^{T}  \left[\left(\frac{1}{2}\right)^t \norm{g_0}^2 + \sum_{i=0}^{t-1} 10d^2L^2\alpha^2 \frac{1}{2^i}  \right]
		+
		8d\sum_{t=1}^{T}\sum_{i=0}^{t-1}\left(\norm{\nabla f_\alpha(x_t)}^2 + \norm{\nabla f_\alpha(x_{t-1})}^2\right)\frac{1}{2^i} \\
		\leq&
		32d \sum_{t=1}^{T} \norm{\nabla f_\alpha(x_t)}^2 + 16d\norm{\nabla f_\alpha(x_0)}^2 +  20d^2L^2\alpha^2 T,
	\end{aligned}
\end{equation}
where the last inequality is also because of  $\norm{g_0} = 0$ implied by  $x_1 = x_0$.

Thus, we can obtain that
\begin{align*}
	\frac{1}{T}\sum_{t=1}^{T}\EE\Big[f_\alpha(x_t) - f_\alpha(x^*)\Big]
	\leq&
	\frac{\norm{x_1 - x^*}^2}{\eta T}
	- \frac{1}{2LT} \sum_{t=1}^{T} \norm{\nabla f_\alpha(x_t)}^2 + \frac{32d\eta}{T} \sum_{t=1}^{T} \norm{\nabla f_\alpha(x_t)}^2\\
	&
	+ 20 \eta d^2L^2\alpha^2
	+ \frac{16\eta d}{T} \norm{\nabla f_\alpha(x_0)}^2\\
	\leq&
	\frac{\norm{x_1 - x^*}^2}{\eta T}
	+ 20 \eta d^2L^2\alpha^2
	+ \frac{16\eta d}{T} \norm{\nabla f_\alpha(x_0)}^2\\
	\leq&
	\frac{\norm{x_1 - x^*}^2}{\eta T}
	+ 20 \eta d^2L^2\alpha^2
	+ \frac{16\eta d L_0^2}{T}, 
\end{align*}
where the second inequality is because of $\eta \leq \frac{1}{64 dL}$ and the last inequality is because  $f(x)$ is $L_0$-Lipschitz continuous.
\end{proof}

\begin{proof}[Proof of Theorem~\ref{thm:main_aa}]

Combining Eq.~\eqref{eq:scx_dd} with the fact $f_\alpha(x) \geq f(x)$ and $f_\alpha(x^*)\leq f(x^*) + \frac{L\alpha^2}{2}$ by Lemma~\ref{lem:sf}, we can obtain that
\begin{align*}
	\frac{1}{T}\sum_{t=1}^{T}\EE\Big[f(x_t) - f(x^*)\Big]
	\leq&
	\frac{1}{T}\sum_{t=1}^{T}\Big(f_\alpha(x_t) - f_\alpha(x^*)\Big) + \frac{L\alpha^2}{2}\\
	\leq&
	\frac{\norm{x_1 - x^*}^2}{\eta T}
	+ 20 \eta d^2L^2\alpha^2
	+ \frac{16\eta d L_0^2}{T} 
	+ \frac{L\alpha^2}{2}.
\end{align*}
\end{proof}

\begin{proof}[Proof of Corollary~\ref{cor:main_ss}]
By the setting of $\alpha$ and the condition of $T$, we can obtain that 
\begin{align*}
	\alpha = \norm{x_1-x^*}^{2/3}\cdot\left(\frac{dL_0}{TL}\right)^{1/3} \leq \frac{L_0}{16dL}.
\end{align*}
Since it holds that $\alpha \leq \frac{L_0}{16dL}$, then we can obtain that $\eta = \min\left\{\frac{\alpha}{4dL_0}, \frac{1}{64dL}\right\} = \frac{\alpha}{4dL_0}$ and $\frac{5dL^2\alpha^3}{L_0} \leq \frac{L\alpha^2}{2}$.
Accordingly, Eq.~\eqref{eq:main_ss} reduces to 
\begin{align*}
	\frac{1}{T}\sum_{t=1}^{T}\Big(f(x_t) - f(x^*)\Big)
	\leq&
	\frac{4dL_0 \norm{x_1 - x^*}^2}{\alpha T}  + \frac{5dL^2\alpha^3}{L_0} + \frac{4L_0\alpha}{T} + \frac{L\alpha^2}{2}\\
	\leq&
	\frac{4dL_0 \norm{x_1 - x^*}^2}{\alpha T}  + L\alpha^2 + \frac{4L_0\alpha}{T}.
\end{align*}
By $\alpha = \norm{x_1-x^*}^{2/3}\cdot\left(\frac{dL_0}{TL}\right)^{1/3}$, we can obtain that
\begin{align*}
	\frac{1}{T}\sum_{t=1}^{T}\Big(f(x_t) - f(x^*)\Big)
	\leq 
	\frac{5L_0^{2/3}L^{1/3}\norm{x_1 - x^*}^{4/3}\cdot d^{2/3}}{T^{2/3}}
	+ \frac{4\cdot d^{1/3} \cdot L_0^{4/3} \norm{x_1 - x^*}^{2/3}}{T^{4/3} L^{1/3}}.
\end{align*}
By setting two terms in the right hand of above equation to $\frac{\varepsilon}{2}$, we can obtain the sample complexity
\begin{align*}
	T = \max\left\{10^{3/2} \norm{x_1-x^*}^2 L L_0^{1/2} \cdot\frac{ d}{\varepsilon^{3/2}},\; \frac{8^{3/4}L_0 \norm{x_1-x^*}^{1/2}}{L^{1/4}}\cdot\frac{d^{1/4}}{\varepsilon^{3/4}}\right\}.
\end{align*}

\end{proof}

We present the following lemmas for the $\mu$-strongly convex case.

\begin{lemma}
\label{lem:tg_ass}
Assume a stochastic gradient $\tg_t$ satisfy  $	\EE\left[\tg_t\right] = \nabla f_\alpha(x_t)$, and 
\begin{equation}\label{eq:srecs_a}
\begin{aligned}
	\EE\left[\norm{\tg_t}^2\right] \leq& 2c_1\cdot d\norm{\nabla f_\alpha(x_t)}^2
	+ 2c_1\cdot d \norm{\nabla f_\alpha(x_{t-1})}^2
	+ \frac{1}{2} \norm{\tg_{t-1}}^2+ 2\zeta,
\end{aligned}
\end{equation}
with $0\leq c_1$ being a constant. 
Letting $f(x)$ be $L$-smooth and $\mu$-strongly convex, then the sequence $\{x_t\}$ generated by the update $x_{t+1} = x_t - \eta \tg_t$ with $\eta \leq  \frac{1}{2c_1 d L}$ satisfies 
\begin{equation}\label{eq:dec}
\begin{aligned}
	\EE\left[f_\alpha(x_{t+1}) - f_\alpha(\tx^*)\right]
	&\leq
	\big(1 - \mu\eta\big) \cdot \Big(f_\alpha(x_t) - f_\alpha(\tx^*)\Big) 
	+ \frac{1}{2}\cdot \frac{L\eta^2}{2}\norm{\tg_{t-1}}^2\\
	&
	+ c_1\cdot dL\eta^2\norm{\nabla f_\alpha(x_{t-1})}^2 
	+ \zeta L\eta^2,
\end{aligned}
\end{equation}
where $\tx^*$ denotes the minimal point of $f_\alpha(x)$.
\end{lemma}
\begin{proof}
Since $f(x)$ is $L$-smooth, $f_\alpha(x)$ is also $L$-smooth by Lemma~\ref{lem:f_alp}.
By the update rule of $\{x_t\}$ and assumption that  $	\EE\left[\tg_t\right] = \nabla f_\alpha(x_t)$, we can obtain that
\begin{equation}\label{eq:ff}
\begin{aligned}
	\EE\left[f_\alpha(x_{t+1}) - f_\alpha(\tx^*)\right]
	\leq& 
	f_\alpha(x_t) -  f_\alpha(\tx^*) - \eta \EE\left[\dotprod{\nabla f_\alpha(x_t), \tg_t}\right] + \frac{L\eta^2}{2}\EE\left[\norm{\tg_t}^2\right]\\
	=&
	f_\alpha(x_t) -  f_\alpha(\tx^*) 
	- \eta \norm{\nabla f_\alpha(x_t)}^2
	+ \frac{L\eta^2}{2}\EE\left[\norm{\tg_t}^2\right].
\end{aligned}
\end{equation}
Combining with Eq.~\eqref{eq:srecs_a}, we can obtain that
\begin{align*}
\EE\left[f_\alpha(x_{t+1}) - f_\alpha(\tx^*)\right]
\stackrel{\eqref{eq:srecs_a}}{\leq}&
f_\alpha(x_t) -  f_\alpha(\tx^*) 
- \eta \norm{\nabla f_\alpha(x_t)}^2
+ c_1\cdot dL\eta^2 \norm{\nabla f_\alpha(x_t)}^2\\
&
+\frac{L\eta^2}{2}\left(\frac{1}{2}\norm{\tg_{t-1}}^2 + 2c_1\cdot d\norm{\nabla f_\alpha(x_{t-1})}^2 + 2\zeta \right)\\
\leq&
f_\alpha(x_t) -  f_\alpha(\tx^*) 
- \frac{\eta}{2} \norm{\nabla f_\alpha(x_t)}^2
+\frac{1}{2}\cdot \frac{L\eta^2}{2}\norm{\tg_{t-1}}^2 \\
&
+ c_1\cdot dL\eta^2\norm{\nabla f_\alpha(x_{t-1})}^2 + \zeta L\eta^2,
\end{align*}
where the last inequality is because of $\eta \leq \frac{1}{2c_1dL}$.
Since $f(x)$ is $\mu$-strongly convex, by Lemma~\ref{lem:f_alp} and Lemma~\ref{lem:mu}, we can obtain
\begin{equation*}
\begin{aligned}
	\EE\left[f_\alpha(x_{t+1}) - f_\alpha(\tx^*)\right]
	\leq&
	\big(1 - \mu\eta\big) \cdot \Big(f_\alpha(x_t) - f_\alpha(\tx^*)\Big) 
	+ \frac{1}{2}\cdot \frac{L\eta^2}{2}\norm{g_{t-1}}^2\\
	&
	+ c_1\cdot dL\eta^2\norm{\nabla f_\alpha(x_{t-1})}^2 
	+ \zeta L\eta^2.
\end{aligned}
\end{equation*}

\end{proof}

\begin{lemma}\label{lem:AA}
Let $f(x)$ and $\tg_t$ satisfy properties in Lemma~\ref{lem:tg_ass}. 
Define $\rho = 1 - \eta \mu$ with $\eta \leq \frac{1}{4L}$ and a sequence $A_i  = \sum_{j=0}^i \frac{1}{(2\rho)^j} $ with $0 \leq i$ and $A_i \leq 3$. 
Denoting $\Delta_t = f_\alpha(x_t) - f_\alpha(\tx^*)$ and setting $\eta \leq \frac{1}{14c_1dL}$, then for any integers $1\leq i \leq t-1$, it holds that 
\begin{equation}\label{eq:recs2}
\begin{aligned}
	&\rho \Delta_{t - i+1} + \frac{L\eta^2}{2} \cdot \frac{A_{i-1}}{2} \EE\left[ \norm{\tg_{t-i}}^2  \right] + 4d A_{i-1} L\eta^2 \norm{\nabla f_\alpha(x_{t-i})}^2\\
	\leq& 
	\rho\left( \rho \Delta_{t - i} + \frac{L\eta^2}{2} \cdot \frac{A_i}{2} \EE\left[ \norm{\tg_{t-i-1}}^2 \right] + 4d A_i  L\eta^2 \norm{\nabla f_\alpha(x_{t-i-1})}^2 + A_i \zeta L\eta^2 \right).
\end{aligned}
\end{equation}
\end{lemma}
\begin{proof}
First, since $\eta \leq \frac{1}{4L}$ and $\mu \leq L$,  then it holds that $\rho = 1 - \mu\eta \geq \frac{3}{4}$.
Furthermore, we can represent $A_i$ and obtain the upper bound of $A_i$ as follows: 
\begin{equation}\label{eq:A_up}
A_{i+1} = 1 + \frac{A_i}{2\rho}, \mbox{ and }\; A_i \leq \frac{1}{1-\frac{1}{2\rho}} \leq \frac{1}{1-\frac{1}{2\cdot \frac{3}{4}}} = 3.
\end{equation}

Furthermore, we can obtain $\eta \leq \frac{1}{2c_1d(A_i + \rho^{-1}A_{i-1})L}$ since it holds that
\begin{equation}\label{eq:eta}
	\eta \leq \frac{1}{14 c_1dL} = \frac{1}{2c_1d (3 + \frac{4}{3} \cdot 3)L} \stackrel{\eqref{eq:A_up}}{\leq}\frac{1}{2c_1d(A_i + \rho^{-1}A_{i-1})L}.
\end{equation}
Accordingly, we have 
\begin{equation}\label{eq:dd}
	\begin{aligned}
		&\Delta_{t - i} - \eta\norm{\nabla f_\alpha(x_{t - i})}^2 + c_1 dL\eta^2 (A_i + \rho^{-1} A_{i-1}) \norm{\nabla f_\alpha(x_{t - i})}^2\\
		\stackrel{\eqref{eq:eta}}{\leq} &
		\Delta_{t - i} - \frac{\eta}{2} \norm{\nabla  f_\alpha(x_{t - i})}^2 
		\leq
		\Delta_{t - i} - \mu\eta \cdot \Delta_{t - i} 
		= 
		\rho \cdot \Delta_{t - i}.
	\end{aligned}
\end{equation}

Similar to the derivation of Eq.~\eqref{eq:dec}, we can obtain that
\begin{align*}
&\rho \Delta_{t - i+1} + \frac{L\eta^2}{2} \cdot \frac{A_{i-1}}{2} \EE\left[ \norm{\tg_{t-i}}^2  \right] + c_1\cdot d A_{i-1} L\eta^2 \norm{\nabla f_\alpha(x_{t-i})}^2\\
\stackrel{\eqref{eq:ff}}{\leq}& 
\rho\left( \Delta_{t - i} - \eta\norm{\nabla f_\alpha(x_{t - i})}^2 + \frac{L\eta^2}{2} \left(1 + \frac{A_{i-1}}{2\rho}\right) \EE\left[\norm{\norm{\tg_{t-i}}^2}\right] +  c_1\rho^{-1} d A_{i-1} L\eta^2 \norm{\nabla f_\alpha(x_{t-i})}^2  \right)\\
\stackrel{\eqref{eq:A_up}}{=}&
\rho\left( \Delta_{t - i} - \eta\norm{\nabla f_\alpha(x_{t - i})}^2 + \frac{L\eta^2}{2} \cdot A_i \cdot \EE\left[\norm{\norm{\tg_{t-i}}^2}\right] +  c_1\rho^{-1} d A_{i-1} L\eta^2 \norm{\nabla f_\alpha(x_{t-i})}^2   \right)\\
\stackrel{\eqref{eq:srecs_a}}{\leq}&
\rho \left( \Delta_{t - i} - \eta\norm{\nabla f_\alpha(x_{t - i})}^2 + c_1 dL\eta^2 (A_i + \rho^{-1} A_{i-1}) \norm{\nabla f_\alpha(x_{t - i})}^2 \right)\\
&+\rho\left( \frac{L\eta^2}{2} \cdot \frac{A_i}{2} \EE\left[ \norm{\tg_{t-i-1}}^2 \right] + c_1 d A_i  L\eta^2 \norm{\nabla f_\alpha(x_{t-i-1})}^2 + A_i\zeta L\eta^2\right)\\
\stackrel{\eqref{eq:dd}}{\leq}&
\rho\left( \rho \Delta_{t - i} + \frac{L\eta^2}{2} \cdot \frac{A_i}{2} \EE\left[ \norm{\tg_{t-i-1}}^2 \right] + c_1 d A_i  L\eta^2 \norm{\nabla f_\alpha(x_{t-i-1})}^2 + A_i\zeta L\eta^2 \right),
\end{align*}
which concludes the proof.
\end{proof}

\begin{proof}[Proof of Theorem~\ref{thm:main}]
The objective function $f(x)$ satisfies the properties required in Lemma~\ref{lem:tg_ass} and Lemma~\ref{lem:AA}.
The gradient estimation $g_t$ also satisfies the properties required in Lemma~\ref{lem:tg_ass} with $c_1 = 4$ and $\zeta = 5d^2L^2\alpha^2$.
Furthermore, the step size $\eta = \min\left\{\frac{\alpha}{4dL_0}, \frac{1}{56dL}\right\}$ satisfies that $\eta \leq \frac{1}{14c_1 dL}$ required in Lemma~\ref{lem:AA}. 
Thus, results of Lemma~\ref{lem:tg_ass} and Lemma~\ref{lem:AA} hold.

Then, by using Eq.~\eqref{eq:recs2} recursively and combining with Lemma~\ref{lem:Delta}, we can obtain that
\begin{align*}
	&\rho \Delta_t + \frac{1}{2} \cdot \frac{L\eta^2}{2} \EE\left[ \norm{g_{t-1}}^2\right] + c_1dL\eta^2 \norm{\nabla f_\alpha(x_{t-1})}^2 + A_i\zeta L\eta^2 \\
	\leq&
	\rho^t \Delta_1 + \rho^{t-1} \frac{L\eta^2}{2} \frac{A_{t-1}}{2} \norm{g_0}^2 + c_1\rho^{t-1}\cdot A_{t-1}dL\eta^2 \norm{\nabla f_\alpha(x_0)}^2
	+\zeta \sum_{i=0}^{t} \rho^i A_i 
	\\
	\leq&
	\rho^t \Delta_1 + \rho^{t-1} \cdot L\eta^2 \norm{g_0}^2 + 3c_1\rho^{t-1}\cdot dL\eta^2 \norm{\nabla f_\alpha(x_0)}^2
	+ \frac{3\zeta L\eta^2}{1-\rho},
\end{align*}
where the last inequality is because of $A_i \leq 3$ and $\sum_{i=0}^{t-1}\rho^iA_i \leq 3 \sum_{i=0}^{t-1}\rho^i \leq 
\frac{3}{1-\rho}$.

Therefore, combining above results and using the fact $\norm{g_0} = 0$ and $\norm{\nabla f_\alpha(x)}^2 \leq L_0^2$, we can obtain that 
\begin{align*}
	\EE\left[ f_\alpha(x_{t+1}) - f_\alpha(\tx^*) \right]
	\leq& 
	\rho^t \Big(f_\alpha(x_1) - f_\alpha(\tx^*)\Big)
	+ 12\rho^{t-1}\cdot dLL_0^2\eta^2  
	+ \frac{3L\eta^2\zeta}{1-\rho}.
\end{align*}
Combining with facts that $f(x) \leq f_\alpha(x) \leq f(x) + \frac{L\alpha^2}{2}$ and $f_\alpha(\tx^*) \geq f(x^*) - \frac{L\alpha^2}{2}$ shown in Lemma~\ref{lem:sf}, we can obtain that
\begin{align*}
	\EE\left[f(x_{t+1}) - f(x^*)\right]
	\leq&
	\rho^t\Big(f(x_1) - f(x^*) + L\alpha^2\Big)
	+ 12\rho^{t-1}\cdot dLL_0^2\eta^2  
	+ \frac{15\eta d^2 L^3 \alpha^2  }{\mu}
	+ L\alpha^2\\
	\leq&
	\rho^t\Big(f(x_1) - f(x^*)\Big)
	+ 12\rho^{t-1}\cdot dLL_0^2\eta^2  
	+ \frac{15\eta d^2 L^3 \alpha^2  }{\mu}
	+ 2L\alpha^2.
\end{align*}
\end{proof}

\begin{proof}[Proof of Corollary~\ref{cor:iter}]
By $\alpha = \left(\frac{\varepsilon}{9L}\right)^{1/2}$ and  $\varepsilon \leq \frac{9L_0^2\mu^2}{14^2 d^2 L^3}$, we can obtain that
\begin{equation*}
	\eta = \min\left\{\frac{\alpha}{4dL_0}, \frac{1}{56dL}\right\} = \frac{\alpha}{4dL_0}, \mbox{ and } \frac{15dL^3\alpha^3}{4L_0\mu} \leq L\alpha^2.
\end{equation*}
Accordingly, Eq.~\eqref{eq:main} reduces to 
\begin{align*}
	\EE\left[ f(x_{T+1}) - f(x^*)  \right]
	\leq&
	\rho^T \Big(f(x_1) - f(x^*)\Big) 
	+ \rho^{T-1}\cdot\frac{3L\alpha^2}{4d}
	+ \frac{15dL^3\alpha^3}{4L_0\mu}
	+ 2L\alpha^2\\
	\leq&
	\rho^T \Big(f(x_1) - f(x^*)\Big) 
	+ \rho^{T-1}\cdot\frac{3L\alpha^2}{4d}
	+ 3L\alpha^2.
\end{align*}
By setting three terms in the right hand of above equation to $\varepsilon/3$, we can obtain that
\begin{align*}
	\alpha^2 = \frac{\varepsilon}{9L}, \mbox{ and } T = \frac{1}{\mu\eta}\log\frac{3\Big(f(x_1) - f(x^*)\Big) }{\varepsilon} = \frac{12dL_0 L^{1/2}}{\mu \varepsilon^{1/2}}\log\frac{3\Big(f(x_1) - f(x^*)\Big) }{\varepsilon}.
\end{align*}


\end{proof}

\section{Proofs of Section~\ref{subsec:nonsmooth}}

\begin{proof}[Proof of Lemma~\ref{lem:g_var_det_non}]
First, by the definition of $g_t$, we can obtain that
\begin{align*}
	\EE\left[\norm{g_t}^2\right] 
	\stackrel{\eqref{eq:gt_def}}{=}& \frac{d^2}{\alpha^2} \EE_{u_t}\left[\Big(f(x_t + \alpha u_t) - f(x_{t-1} +\alpha u_{t-1})\Big)^2 \norm{u_t}^2\right]\\
	\leq&
	\frac{3d^2}{\alpha^2} \EE_{u_t}\left[\Big(f(x_t + \alpha u_t) - \EE_{u_t}[f(x_t +\alpha u_t)]\Big)^2 \norm{u_t}^2\right]\\
	+&\frac{3d^2}{\alpha^2} \EE_{u_t}\left[\EE_{u_{t-1}}\left[\Big(f(x_{t-1} + \alpha u_{t-1}) - \EE_{u_{t-1}}[f(x_{t-1} + \alpha u_{t-1})]\Big)^2\right] \norm{u_t}^2\right]\\
	+&\frac{3d^2}{\alpha^2} \EE_{u_t}\left[\Big(\EE_{u-1}\left[f(x_{t-1} + \alpha u_{t-1})\right] - \EE_{u_t}[f(x_t +\alpha u_t)]\Big)^2 \norm{u_t}^2\right].
\end{align*}
Next, we will bound the three terms on the right-hand side of the above equations.
We first have
\begin{align*}
	&\EE_{u_t}\left[\Big(f(x_t + \alpha u_t) - \EE_{u_t}[f(x_t +\alpha u_t)]\Big)^2 \norm{u_t}^2\right]\\
	\leq&
	\sqrt{\EE_{u_t}\left[\Big(f(x_t + \alpha u_t) - \EE_{u_t}[f(x_t +\alpha u_t)]\Big)^4 \right]} \cdot \sqrt{\EE_{u_t}\left[\norm{u_t}^4\right]}\\
	\leq&
	\sqrt{\EE_{u_t}\left[\Big(f(x_t + \alpha u_t) - \EE_{u_t}[f(x_t +\alpha u_t)]\Big)^4 \right]},
\end{align*}
where the first inequality is because of the Cauchy's inequality.
Since $f(x)$ is $L_0$-Lipschitz continuous, then $f(x_t + \alpha u)$ is $(\alpha L_0)$-Lipschitz continuous with respect to $u$.
By Lemma~\ref{lem:g_var}, we can obtain that
\begin{equation*}
	\sqrt{\EE_{u_t}\left[\Big(f(x_t + \alpha u_t) - \EE_{u_t}[f(x_t +\alpha u_t)]\Big)^4 \right]} 
	\leq 
	\frac{ c_0L_0^2\alpha^2 }{d}.
\end{equation*}

Similarly, we can obtain that
\begin{align*}
	&\EE_{u_t}\left[\EE_{u_{t-1}}\left[\Big(f(x_{t-1} + \alpha u_{t-1}) - \EE_{u_{t-1}}[f(x_{t-1} + \alpha u_{t-1})]\Big)^2\right] \norm{u_t}^2\right] \\
	\leq &
	\sqrt{\EE_{u_{t-1}}\left[\Big(f(x_{t-1} + \alpha u_{t-1}) - \EE_{u_{t-1}}[f(x_{t-1} +\alpha u_{t-1})]\Big)^4 \right]}
	\leq 
	\frac{ c_0L_0^2\alpha^2 }{d}
\end{align*}

Finally, we can obtain that 
\begin{align*}
	&\EE_{u_t}\left[\Big(\EE_{u-1}\left[f(x_{t-1} + \alpha u_{t-1})\right] - \EE_{u_t}[f(x_t +\alpha u_t)]\Big)^2 \norm{u_t}^2\right]\\
	=&
	\EE_{u_t}\left[\Big(\EE_u\left[f(x_{t-1} + \alpha u)\right] - \EE_u\left[f(x_t + \alpha u)\right]  \Big)^2 \norm{u_t}^2\right]\\
	=& 
	\Big(\EE_u\left[f(x_{t-1} + \alpha u)\right] - \EE_u\left[f(x_t + \alpha u)\right]  \Big)^2\\
	\leq&
	\EE_u\left[\Big(f(x_{t-1} + \alpha u) - f(x_t + \alpha u)\Big)^2\right]\\
	\leq&
	L_0^2|x_t - x_{t-1}|^2\\
	=&
	L_0^2\eta^2 \norm{g_{t-1}}^2,
\end{align*}
where the first inequality is because of the Jensen's inequality and second one is because of $L_0$-Lipschitz continuity of $f(x)$.

Combining above equations with $\eta \leq \frac{\alpha}{3dL_0}$, we can obtain that
\begin{align*}
	\EE\left[\norm{g_t}^2\right]
	\leq 
	6c_0 d L_0^2 + \frac{3L_0^2d^2 \eta^2}{\alpha^2} \norm{g_{t-1}}^2
	\leq
	6c_0 d L_0^2 + \frac{1}{2} \norm{g_{t-1}}^2.
\end{align*} 
Combining  above equation and  Lemma~\ref{lem:Delta} with $\rho = 1/2$ and $M_t = 6c_0dL_0^2$, we can obtain that
\begin{align*}
\EE\left[\norm{g_t}^2\right] 
\leq
\left[\left(\frac{1}{2}\right)^t \norm{g_0}^2 + \sum_{i=0}^{t-1} 6c_0dL_0^2 \frac{1}{2^i}  \right]
\leq
12 c_0 d L_0^2,
\end{align*}
where the last inequality is also because of   $\norm{g_0} = 0$ implied by $x_0 = x_1$.
\end{proof}

\begin{proof}[Proof of Theorem~\ref{thm:main_ns}]
By the update rule of Algorithm~\ref{alg:CG}, we can obtain that
\begin{equation}\label{eq:xxx}
\begin{aligned}
	&\EE\left[\norm{x_{t+1} - x^*}^2\right]
	=
	\EE\left[\norm{x_t - \eta g_t - x^*}^2\right]\\
	=&
	\norm{x_t - x^*}^2 - \eta \EE\left[\dotprod{g_t, x_t - x^*}\right] + \eta^2 \EE\left[\norm{g_t}^2\right]\\
	=&
	\norm{x_t - x^*}^2 - \eta \dotprod{\nabla f_\alpha(x_t), x_t - x^*} + \eta^2 \EE\left[\norm{g_t}^2\right]\\
	\leq&
	\norm{x_t - x^*} - \eta \Big(f_\alpha(x_t) - f_\alpha(x^*)\Big) + \eta^2\EE\left[\norm{g_t}^2\right],
\end{aligned}
\end{equation}
where the third equality is because of the fact that $\EE\left[g_t\right] = \nabla f_\alpha(x)$ shown in Lemma~\ref{lem:f_alp} and the last inequality is because $f(x)$ is convex.

We represent above equation and sum it from $t=1$ to $T$.
We can obtain that
\begin{equation*}
\begin{aligned}
	\sum_{t=1}^{T} \eta \EE\left[f_\alpha(x_t) - f_\alpha(x^*)\right] 
	\leq& 
	\sum_{t=1}^{T} \left[ \norm{x_t - x^*}^2 - \norm{x_{t+1} - x^*}^2 \right]
	+
	\sum_{t=1}^{T} \eta^2 \EE\left[\norm{g_t}^2\right]\\
	=&
	\norm{x_1 - x^*}^2 - \norm{x_{T+1} - x^*}^2 + \eta^2 \sum_{t=1}^{T} \EE\left[\norm{g_t}^2\right]\\
	\leq&
	\norm{x_1 - x^*}^2 + \eta^2\sum_{t=1}^{T} \EE\left[\norm{g_t}^2\right]\\
	\stackrel{\eqref{eq:g_dec}}{\leq}&
	\norm{x_1 - x^*}^2 + 12  c_0 d \eta^2 L_0^2 T.
\end{aligned}
\end{equation*}
Dividing  $\frac{1}{T\eta}$ to both sides of above equation, we can obtain 
\begin{align*}
\frac{1}{T}\sum_{t=1}^{T} \Big(f_\alpha(x_t) - f_\alpha(x^*)\Big)
\leq 
\frac{\norm{x_1 - x^*}^2}{T\eta} + 12\eta c_0 dL_0^2.
\end{align*}

Since it holds that $f_\alpha(x) \ge f(x)$ and $f_\alpha(x) \leq f(x) + L_0\alpha$ by Lemma~\ref{lem:sf}, then we can obtain that
\begin{align*}
\frac{1}{T}\sum_{t=1}^{T} \Big(f(x_t) - f(x^*)\Big)
\leq
\frac{1}{T}\sum_{t=1}^{T} \Big(f_\alpha(x_t) - f_\alpha(x^*)\Big) + L_0\alpha
\leq 
\frac{\norm{x_1 - x^*}^2}{T\eta} + 12\eta c_0 dL_0^2 + L_0\alpha.
\end{align*}

By setting  $\alpha = d^{1/2} T^{-1/2}$ and $\eta = \frac{\alpha}{3d L_0} = \frac{1}{3d^{1/2}T^{1/2}L_0}$, we can obtain that
\begin{align*}
\frac{1}{T}\sum_{t=1}^{T} \Big(f(x_t) - f(x^*)\Big)
\leq 
\frac{3\norm{x_1 - x^*}^2 d^{1/2} L_0}{T^{1/2}} 
+ \frac{4c_0 L_0d^{1/2} }{T^{1/2}} + \frac{d^{1/2}L_0}{T^{1/2}}.
\end{align*} 
\end{proof}

\begin{proof}[Proof of Corollary~\ref{cor:ns_s}]
To find an $\varepsilon$-suboptimal solution, we only require the right side of Eq.~\eqref{eq:ns} to be $\varepsilon$.
Letting $\frac{3\norm{x_1 - x^*}^2d^{1/2} L_0}{T^{1/2}} = \varepsilon/3$, then $T$ only needs to be
$
T = 9^2 \norm{x_1 - x^*}^2 L_0^2\cdot \frac{d}{\varepsilon^2}.
$
Letting $\frac{4c_0L_0 d^{1/2}}{T^{1/2}} = \varepsilon/3$, then $T$ only needs to be
$
T = 12^2c_0^2 L_0^2 \cdot \frac{d}{\varepsilon^2}.
$
Letting $\frac{L_0 d^{1/2}}{T^{1/2}} = \varepsilon/3$, then $T$ only needs to be
$
T = 9 L_0^2 \cdot \frac{d}{\varepsilon^2}.
$

Combining above results, we can obtain that 
\begin{align*}
	T = \max\left\{9^2\norm{x_1 - x^*}^2L_0^2,\; 12^2c_0^2 L_0^2,\; 9L_0^2\right\} \cdot \frac{d}{\varepsilon^2}.
\end{align*}
Since for each iteration our algorithm only queries one function value, the sample complexity of our algorithm is also $T$. 
\end{proof}

\begin{proof}[Proof of Theorem~\ref{thm:main_ns_mu}]
Similar to Eq.~\eqref{eq:xxx}, we can obtain that
\begin{align*}
	\EE\left[\norm{x_{t+1} - x^*}^2\right]
	=&
	\norm{x_t - x^*}^2 - \eta \dotprod{\nabla f_\alpha(x_t), x_t - x^*} + \eta^2 \EE\left[\norm{g_t}^2\right]\\
	\leq&
	\norm{x_t - x^*}^2 - \eta \Big(f_\alpha(x_t) - f_\alpha(x^*)\Big) 
	- \frac{\eta\mu}{2} \norm{x_t - x^*}^2 
	+ \eta^2 \EE\left[\norm{g_t}^2\right]\\
	=&
	\left(1 - \frac{\mu\eta}{2}\right) \norm{x_t - x^*}^2 
	- \eta \Big(f_\alpha(x_t) - f_\alpha(x^*)\Big)
	+ \eta^2 \EE\left[\norm{g_t}^2\right]\\
	\stackrel{\eqref{eq:g_dec}}{\leq}&
	\left(1 - \frac{\mu\eta}{2}\right) \norm{x_t - x^*}^2 
	- \eta \Big(f_\alpha(x_t) - f_\alpha(x^*)\Big)
	+ 12c_0 d\eta^2L_0^2,
\end{align*}
where the first inequality is because of $f_\alpha(x)$ is also $\mu$-strongly convex by Lemma~\ref{lem:f_alp}.

Using the notation $\rho$, and we represent above equation as follows:
\begin{align*}
	\EE\Big[f_\alpha(x_t) - f_\alpha(x^*)\Big]
	\leq
	\frac{\rho \EE\left[\norm{x_t - x^*}^2\right]}{\eta} - \frac{\EE\left[\norm{x_{t+1} - x^*}^2\right]}{\eta} + 12c_0d\eta L_0^2.
\end{align*} 
Thus, we have
\begin{align*}
	\sum_{t=1}^{T} \EE\left[ w_t\Big(f_\alpha(x_t) - f_\alpha(x^*)\Big)\right]
	\leq& 
	\sum_{t=1}^{T}\left(\frac{w_t\rho \norm{x_t - x^*}^2}{\eta} - \frac{w_t\norm{x_{t+1} - x^*}^2}{\eta}\right) 
	+ 12c_0d\eta L_0^2 \sum_{t=1}^{T}w_t\\
	=&
	\sum_{t=1}^{T}\left(\frac{w_{t-1} \norm{x_t - x^*}^2}{\eta} - \frac{w_t\norm{x_{t+1} - x^*}^2}{\eta}\right) 
	+ 12c_0d\eta L_0^2 \sum_{t=1}^{T}w_t\\
	=&
	\frac{w_0 \norm{x_1 - x^*}^2}{\eta} - \frac{w_T\norm{x_{T+1} - x^*}^2}{\eta}
	+ 12c_0d\eta L_0^2 \sum_{t=1}^{T}w_t\\
	\leq&
	\frac{w_0 \norm{x_1 - x^*}^2}{\eta} 
	+ 12c_0d\eta L_0^2 \sum_{t=1}^{T}w_t,
\end{align*}
where the first equality is because of $w_t = \rho^{-1} w_{t-1}$ and the last equality is because of telescoping terms.

Combining with the facts $f_\alpha(x) \ge f(x)$ and $f_\alpha(x) \leq f(x) + L_0\alpha$ by Lemma~\ref{lem:sf}, we can obtain that
\begin{align*}
	\frac{1}{\sum_{t=1}^{T} w_t}\sum_{t=1}^{T} \EE\left[w_t\left(f(x_t) - f(x^*)\right)\right]
	\leq&
	\frac{1}{\sum_{t=1}^{T} w_t}\sum_{t=1}^{T} \EE\left[w_t\left(f_\alpha(x_t) - f_\alpha(x^*)\right)\right] + L_0\alpha\\
	\leq&
	\frac{w_0\norm{x_t - x^*}^2}{\eta \sum_{t=1}^{T} w_t}
	+
	12c_0d\eta L_0^2 + L_0\alpha\\
	=&
	\frac{\mu \norm{x_1 - x^*}^2}{2\left(\rho^{-T} - 1\right)}
	+ 
	12c_0d\eta L_0^2 + L_0\alpha\\
	=&
	\frac{\mu \norm{x_1 - x^*}^2}{2\left(\rho^{-T} - 1\right)}
	+ 
	(4c_0+1)L_0\alpha,	
\end{align*}
where the first equality is because of 
\begin{align*}
	\sum_{t=1}^{T}w_t 
	= \frac{\rho^{-1} \left(\rho^{-T} - 1\right)}{\rho^{-1} - 1}
	= \frac{ \left(\rho^{-T} - 1\right)}{1 - \rho}
	= \frac{ \left(\rho^{-T} - 1\right)}{\mu\eta/2}.
\end{align*}

Finally, by the convexity of $f(x)$, we can obtain that
\begin{align*}
	f\left(\frac{\sum_{t=1}^{T}w_t x_t}{\sum_{t=1}^{T}w_t}\right) - f(x^*) \leq \frac{1}{\sum_{t=1}^{T} w_t}\sum_{t=1}^{T} w_t\left(f(x_t) - f(x^*)\right).
\end{align*}
Combining above results, we can conclude the proof.

\end{proof}

\begin{proof}[Proof of Corollary~\ref{cor:main_ns_mu}]
To find an $\varepsilon$-suboptimal solution, we only require the right side of Eq.~\eqref{eq:ns_ss} to be $\varepsilon$.
By setting $(4c_0+1)L_0\alpha = \frac{\varepsilon}{2}$, we can obtain that $\alpha = \frac{\varepsilon}{2(4c_0+1)L_0}$.
Since $\eta = \alpha/(3dL_0)$, then it holds that $\eta = \frac{\varepsilon}{6(4c_0+1)dL_0^2}$.
To achieve $\frac{\mu\norm{x_1 - x^*}^2}{2(\rho^{-T} - 1)} \leq \frac{\varepsilon}{2}$, we only require that $\rho^{-T} \geq \frac{\mu\norm{x_1 - x^*}^2}{\varepsilon} + 1$ which equals to $\rho^T \leq  \left(\frac{\mu\norm{x_1 - x^*}^2}{\varepsilon} + 1\right)^{-1}$.
By the value of $T$, we can validate that
\begin{align*}
	\rho^T  = (1-\mu\eta)^T 
	\leq 
	\exp\left(-T\mu\eta\right)
	=
	\exp\left(-\log\left(\frac{\mu\norm{x_1 - x^*}^2}{\varepsilon} + 1\right)\right) 
	= \left(\frac{\mu\norm{x_1 - x^*}^2}{\varepsilon} +  1\right)^{-1},
\end{align*}
which concludes the proof.
\end{proof}

\section{Proofs of Section~\ref{subsec:snc}}

\begin{proof}[Proof of Theorem~\ref{thm:main_nc}]
	We reformulate Eq.~\eqref{eq:ff} as follows
	\begin{equation*}
		\eta \norm{\nabla f_\alpha(x_t)}^2 \leq f_\alpha(x_t) - \EE_{u_t} \left[f_\alpha(x_{t+1})\right] + \frac{L\eta^2}{2} \EE\left[\norm{g_t}^2\right].
	\end{equation*}
	Summing up above equation from $1$ to $T$ and telescoping terms, we can obtain that
	\begin{align*}
		&\eta\sum_{t=1}^{T}\EE\left[\norm{\nabla f_\alpha(x_t)}^2\right]\\
		\leq& 
		f_\alpha(x_1) - f_\alpha(x_{T+1}) + \frac{L\eta^2}{2}\sum_{t=1}^{T}\EE \left[\norm{g_t}^2\right]\\
		\stackrel{\eqref{eq:gt_sum}}{\leq}&
		f_\alpha(x_1) - f_\alpha(x_{T+1}) 
		+ 8dL\eta^2 \sum_{t=1}^{T} \norm{\nabla f_\alpha(x_t)}^2
		+ 8dL\eta^2 \norm{\nabla f_\alpha(x_0)}^2 
		+ 10d^2 \eta^2L^3\alpha^2T\\
		\leq&
		f_\alpha(x_1) - f_\alpha(x_{T+1}) 
		+ 8dL\eta^2 \sum_{t=1}^{T} \norm{\nabla f_\alpha(x_t)}^2
		+ 8dLL_0^2\eta^2  
		+ 10d^2 \eta^2L^3\alpha^2T.
	\end{align*} 
	The above equation can be represented as
	\begin{align*}
		\eta(1 - 8dL\eta) \sum_{t=1}^{T}\EE\left[\norm{\nabla f_\alpha(x_t)}^2\right] 
		\leq 
		f_\alpha(x_1) - f_\alpha(x_{T+1})
		+ 8dLL_0^2\eta^2  
		+ 10d^2 \eta^2L^3\alpha^2T.
	\end{align*}
	Furthermore, by $\eta \leq \frac{1}{16dL}$, it holds that $\frac{\eta}{2} \leq \eta(1 - 8dL\eta) $.
	Thus,
	\begin{align*}
		\frac{1}{T}\sum_{t=1}^{T}\EE\left[\norm{\nabla f_\alpha(x_t)}^2\right] 
		=& 
		\frac{2}{\eta T} \cdot \frac{\eta}{2} \sum_{t=1}^{T}\EE\left[\norm{\nabla f_\alpha(x_t)}^2\right]\\
		\leq&
		\frac{2}{\eta T} \cdot \eta(1 - 8dL\eta) \sum_{t=1}^{T}\EE\left[\norm{\nabla f_\alpha(x_t)}^2\right]\\ 
		\leq&
		\frac{2\left(f_\alpha(x_1) - f_\alpha(x_{T+1})\right)}{\eta T} 
		+ \frac{16dLL_0^2 \eta}{T}
		+ 20 \eta d^2 L^3 \alpha^2.
	\end{align*}
	By Lemma~\ref{lem:sf}, we can obtain that
	\begin{align*}
		\norm{\nabla f(x_t)}^2 
		=&
		\norm{\nabla f_\alpha(x_t) + \nabla f(x_t) - \nabla f_\alpha(x_t)}^2\\
		\leq&
		2\norm{\nabla f_\alpha(x_t)}^2 + 2 \norm{\nabla f(x_t) - \nabla f_\alpha(x_t)}^2\\
		\leq&
		2\norm{\nabla f_\alpha(x_t)}^2 + 2L^2\alpha^2. 
	\end{align*}
	Also by Lemma~\ref{lem:sf}, we have
	\begin{align*}
		f_\alpha(x_1) - f_\alpha(x_{T+1})
		\leq 
		f_\alpha(x_1) - f_\alpha(\tx^*)
		\leq
		f(x_1) + \frac{L\alpha^2}{2} - f(x^*) + \frac{L\alpha^2}{2}
		=
		f(x_1) - f(x^*) + L\alpha^2.
	\end{align*}
	
	Combining above results, we can obtain that
	\begin{align*}
		\frac{1}{T}\sum_{t=1}^{T}\EE\left[\norm{\nabla f(x_t)}^2\right]
		\leq&
		2\cdot \frac{1}{T}\sum_{t=1}^{T}\EE\left[\norm{\nabla f_\alpha(x_t)}^2\right] + 2L^2\alpha^2\\
		\leq&
		\frac{4\left(f_\alpha(x_1) - f_\alpha(x_{T+1})\right)}{\eta T} 
		+ \frac{32dLL_0^2 \eta}{T}
		+ 40 \eta d^2 L^3 \alpha^2
		+ 2L^2\alpha^2\\
		\leq&
		\frac{4\left(f(x_1) - f(x^*)\right)}{\eta T} 
		+ \frac{4L\alpha^2}{\eta T}
		+ \frac{32dLL_0^2 \eta}{T}
		+ 40 \eta d^2 L^3 \alpha^2
		+ 2L^2\alpha^2.
	\end{align*}
\end{proof}

\begin{proof}[Proof of Corollary~\ref{cor:main_nc}]
	By the setting $\alpha = \left(\frac{dL_0 (f(x_1) - f(x^*))}{L^2T}\right)^{1/3}$ and the condition of $T$ in Eq.~\eqref{eq:T_cond}, we can obtain that 
	\begin{align*}
		\alpha \leq \frac{L_0}{4d L}, \quad \alpha < \frac{L_0}{5d}, \quad\mbox{ and } \quad\alpha \geq \frac{6dL_0}{LT}.
	\end{align*}
	We can further obtain that
	\begin{align*}
		\eta = \min\left\{\frac{\alpha}{4dL_0}, \frac{1}{16dL} \right\} = \frac{\alpha}{4dL_0},  \quad\frac{10dL^2\alpha^3}{L_0} < 2L^2\alpha^2, \quad\mbox{and} \quad\frac{24d L_0 L \alpha}{T} < 4L^2\alpha^2.
	\end{align*}
	By $\eta = \frac{\alpha}{4 d L_0}$, Eq.~\eqref{eq:main_nc} can be represented as 
	\begin{align*}
		\frac{1}{T}\sum_{t=1}^{T}\EE\left[\norm{\nabla f(x_t)}^2\right]
		\leq&
		\frac{16dL_0(f(x_1) - f(x^*))}{T\alpha} 
		+ \frac{16d L_0 L \alpha}{T}
		+ \frac{8d L_0 L \alpha}{T} 
		+ \frac{10 dL^2\alpha^3}{L_0}
		+ 2L^2\alpha^2 \notag\\
		=&
		\frac{16dL_0(f(x_1) - f(x^*))}{T\alpha} 
		+ \frac{24d L_0 L \alpha}{T}
		+ \frac{10 dL^2\alpha^3}{L_0}
		+ 2L^2\alpha^2\\
		\leq&
		\frac{16dL_0(f(x_1) - f(x^*))}{T\alpha} 
		+ \frac{24d L_0 L \alpha}{T}
		+ 4L^2\alpha^2\\
		\leq&
		\frac{16dL_0(f(x_1) - f(x^*))}{T\alpha} 
		+ 8L^2\alpha^2,
	\end{align*}
	where the second and last inequalities are because of $\frac{10dL^2\alpha^3}{L_0} < 2L^2\alpha^2$ and $\frac{24d L_0 L \alpha}{T} < 4L^2\alpha^2$, respectively.
	
	By $\alpha = \left(\frac{dL_0 (f(x_1) - f(x^*))}{L^2T}\right)^{1/3}$, we can obtain that
	\begin{align*}
		\frac{1}{T}\sum_{t=1}^{T}\EE\left[\norm{\nabla f(x_t)}^2\right]
		\leq 24\left(\frac{dLL_0  (f(x_1) - f(x^*))}{T}\right)^{2/3}.
	\end{align*}
	Thus, to achieve $\frac{1}{T}\sum_{t=1}^{T}\EE\left[\norm{\nabla f(x_t)}^2\right] \leq \varepsilon^2$, $T$ only needs to satisfy that
	\begin{equation*}
		T = \frac{24^{3/2} d L_0L (f(x_1) - f(x^*))}{\varepsilon^3}.
	\end{equation*}

\end{proof}

\section{Proof of Section~\ref{subsec:sns}}

\begin{proof}[Proof of Lemma~\ref{lem:sgt_dec}]
	First, by the definition of $\tg_t$, we can obtain that
	\begin{align*}
		\EE\left[\norm{\tg_t}^2\right] 
		\stackrel{\eqref{eq:sgt_def}}{=}& \frac{d^2}{\alpha^2} \EE\left[\Big(f(x_t + \alpha u_t, \xi_t) - f(x_{t-1} +\alpha u_{t-1}, \xi_{t-1})\Big)^2 \norm{u_t}^2\right]\\
		\leq&
		\frac{3d^2}{\alpha^2} \EE\left[\Big(f(x_t + \alpha u_t, \xi_t) - \EE_{u_t}[f(x_t +\alpha u_t)]\Big)^2 \norm{u_t}^2\right]\\
		+&\frac{3d^2}{\alpha^2} \EE\left[\EE_{u_{t-1}}\left[\Big(f(x_{t-1} + \alpha u_{t-1}, \xi_{t-1}) - \EE_{u_{t-1}}[f(x_{t-1} + \alpha u_{t-1})]\Big)^2\right] \norm{u_t}^2\right]\\
		+&\frac{3d^2}{\alpha^2} \EE\left[\Big(\EE_{u-1}\left[f(x_{t-1} + \alpha u_{t-1})\right] - \EE_{u_t}[f(x_t +\alpha u_t)]\Big)^2 \norm{u_t}^2\right].
	\end{align*}
	Next, we will bound the three terms on the right-hand side of the above equation.
	We first have
	\begin{align*}
		&\EE\left[\Big(f(x_t + \alpha u_t, \xi_t) - \EE_{u_t}[f(x_t +\alpha u_t)]\Big)^2 \norm{u_t}^2\right]\\
		\leq&
		2\EE\left[\Big(f(x_t + \alpha u_t, \xi_t) - f(x_t+\alpha u_t)\Big)^2 \cdot \norm{u_t}^2\right] 
		+2 \EE\left[\Big(f(x_t + \alpha u_t) - \EE_{u_t}\left[f(x_t+\alpha u_t)\right]\Big)^2 \cdot \norm{u_t}^2\right] \\
		\leq&
		2\sigma_0^2 + 2 \EE\left[\Big(f(x_t + \alpha u_t) - \EE_{u_t}\left[f(x_t+\alpha u_t)\right]\Big)^2 \right]\\
		\leq&
		2\sigma_0^2 + 2\sqrt{\EE_{u_t}\left[\Big(f(x_t + \alpha u_t) - \EE_{u_t}[f(x_t +\alpha u_t)]\Big)^4 \right]},
	\end{align*}
	where the second inequality is because of  Assumption~\ref{ass:bnd_val_var} and the last inequality is because of the Cauchy's inequality.
	Since $f(x)$ is $L_0$-Lipschitz continuous, then $f(x_t + \alpha u)$ is $(\alpha L_0)$-Lipschitz continuous with respect to $u$.
	By Lemma~\ref{lem:g_var}, we can obtain that
	\begin{equation*}
		\sqrt{\EE_{u_t}\left[\Big(f(x_t + \alpha u_t) - \EE_{u_t}[f(x_t +\alpha u_t)]\Big)^4 \right]} 
		\leq 
		\frac{ c_0L_0^2\alpha^2 }{d}.
	\end{equation*}
	
	Similarly, we can obtain that
	\begin{align*}
		&\EE\left[\EE_{u_{t-1}}\left[\Big(f(x_{t-1} + \alpha u_{t-1}) - \EE_{u_{t-1}}[f(x_{t-1} + \alpha u_{t-1})]\Big)^2\right] \norm{u_t}^2\right] \\
		\leq &
		\sqrt{\EE_{u_{t-1}}\left[\Big(f(x_{t-1} + \alpha u_{t-1}) - \EE_{u_{t-1}}[f(x_{t-1} +\alpha u_{t-1})]\Big)^4 \right]}
		\leq 
		\frac{ c_0L_0^2\alpha^2 }{d}.
	\end{align*}
	Thus, we can obtain that
	\begin{align*}
		&\EE\left[\EE_{u_{t-1}}\left[\Big(f(x_{t-1} + \alpha u_{t-1}, \xi_{t-1}) - \EE_{u_{t-1}}[f(x_{t-1} + \alpha u_{t-1})]\Big)^2\right] \norm{u_t}^2\right]\\
		\leq&
		2\EE\left[\Big(f(x_{t-1} + \alpha u_{t-1}, \xi_{t-1}) - f(x_{t-1}, \alpha u_{t-1})\Big)^2\cdot\norm{u_t}^2\right]\\
		+&
		2\EE\left[\Big(f(x_{t-1} + \alpha u_{t-1}) - \EE_{u_{t-1}}[f(x_{t-1} + \alpha u_{t-1})]\Big)^2 \norm{u_t}^2\right]\\
		\leq&
		2\sigma_0^2 + \frac{ 2c_0L_0^2\alpha^2 }{d}.
	\end{align*}
	
	Finally, we can obtain that 
	\begin{align*}
		&\EE_{u_t}\left[\Big(\EE_{u-1}\left[f(x_{t-1} + \alpha u_{t-1})\right] - \EE_{u_t}[f(x_t +\alpha u_t)]\Big)^2 \norm{u_t}^2\right]\\
		=&
		\EE_{u_t}\left[\Big(\EE_u\left[f(x_{t-1} + \alpha u)\right] - \EE_u\left[f(x_t + \alpha u)\right]  \Big)^2 \norm{u_t}^2\right]\\
		=& 
		\Big(\EE_u\left[f(x_{t-1} + \alpha u)\right] - \EE_u\left[f(x_t + \alpha u)\right]  \Big)^2\\
		\leq&
		\EE_u\left[\Big(f(x_{t-1} + \alpha u) - f(x_t + \alpha u)\Big)^2\right]\\
		\leq&
		L_0^2|x_t - x_{t-1}|^2\\
		=&
		L_0^2\eta^2 \norm{\tg_{t-1}}^2,
	\end{align*}
	where the first inequality is because of the Jensen's inequality and second one is because of $L_0$-Lipschitz continuity of $f(x)$.
	
	Combining above results, we can obtain that
	\begin{align*}
		\EE\left[\norm{\tg_t}^2\right]
		\leq
		\frac{3d^2L_0^2\eta^2}{\alpha^2} \norm{\tg_{t-1}}^2 
		+
		\frac{12d^2\sigma_0^2}{\alpha^2} + 12c_0 dL_0^2.
	\end{align*}
	
	By $\eta \leq \frac{\alpha}{3dL_0}$, we can obtain that
	\begin{align*}
		\EE\left[\norm{\tg_t}^2\right]
		\leq
		\frac{1}{2} \norm{\tg_{t-1}}^2 + 12c_0 d L_0^2 + \frac{12d^2\sigma_0^2}{\alpha^2}.
	\end{align*} 
	Combining  above equation and  Lemma~\ref{lem:Delta} with $\rho = 1/2$ and $M_t = 6c_0dL_0^2$, we can obtain that
	\begin{align*}
		\EE\left[\norm{\tg_t}^2\right] 
		\leq
		\left(\frac{1}{2}\right)^t \norm{\tg_0}^2 + \sum_{i=0}^{t-1} \left( 12c_0dL_0^2 +\frac{12d^2\sigma_0^2}{\alpha^2}\right)\frac{1}{2^i}   
		\leq
		24 c_0 d L_0^2  + \frac{24d^2\sigma_0^2}{\alpha^2} ,
	\end{align*}
	where the last inequality is also because of   $\norm{\tg_0} = 0$ implied by $x_0 = x_1$.
\end{proof}

\begin{proof}[Proof of Theorem~\ref{thm:main_nss}]
	By the update rule of Algorithm~\ref{alg:SCG}, we can obtain that
	\begin{equation}\label{eq:x4}
		\begin{aligned}
			&\EE\left[\norm{x_{t+1} - x^*}^2\right]
			=
			\EE\left[\norm{x_t - \eta \tg_t - x^*}^2\right]\\
			=&
			\norm{x_t - x^*}^2 - \eta \EE\left[\dotprod{\tg_t, x_t - x^*}\right] + \eta^2 \EE\left[\norm{\tg_t}^2\right]\\
			=&
			\norm{x_t - x^*}^2 - \eta \dotprod{\nabla f_\alpha(x_t), x_t - x^*} + \eta^2 \EE\left[\norm{\tg_t}^2\right]\\
			\leq&
			\norm{x_t - x^*} - \eta \Big(f_\alpha(x_t) - f_\alpha(x^*)\Big) + \eta^2\EE\left[\norm{\tg_t}^2\right],
		\end{aligned}
	\end{equation}
	where the third equality is because of the fact that $\EE\left[g_t\right] = \nabla f_\alpha(x)$ shown in Lemma~\ref{lem:f_alp} and the last inequality is because $f(x)$ is convex.
	
	We represent above equation and sum it from $t=1$ to $T$.
	We can obtain that
	\begin{equation*}
		\begin{aligned}
			\sum_{t=1}^{T} \eta \EE\left[f_\alpha(x_t) - f_\alpha(x^*)\right] 
			\leq& 
			\sum_{t=1}^{T} \left[ \norm{x_t - x^*}^2 - \norm{x_{t+1} - x^*}^2 \right]
			+
			\sum_{t=1}^{T} \eta^2 \EE\left[\norm{\tg_t}^2\right]\\
			=&
			\norm{x_1 - x^*}^2 - \norm{x_{T+1} - x^*}^2 + \eta^2 \sum_{t=1}^{T} \EE\left[\norm{\tg_t}^2\right]\\
			\leq&
			\norm{x_1 - x^*}^2 + \eta^2\sum_{t=1}^{T} \EE\left[\norm{\tg_t}^2\right]\\
			\stackrel{\eqref{eq:sg_dec}}{\leq}&
			\norm{x_1 - x^*}^2 + 	24 c_0 d L_0^2\eta^2 T  + \frac{24d^2\sigma_0^2 \eta^2 T}{\alpha^2}.
		\end{aligned}
	\end{equation*}
	Dividing  $\frac{1}{T\eta}$ to both sides of above equation, we can obtain 
	\begin{align*}
		\frac{1}{T}\sum_{t=1}^{T} \EE\Big[f_\alpha(x_t) - f_\alpha(x^*)\Big]
		\leq 
		\frac{\norm{x_1 - x^*}^2}{T\eta} + 24\eta c_0 dL_0^2 + \frac{24\eta d^2\sigma_0^2}{\alpha^2}.
	\end{align*}
	Since it holds that $f_\alpha(x) \ge f(x)$ and $f_\alpha(x) \leq f(x) + L_0\alpha$ by Lemma~\ref{lem:sf}, then we can obtain that
	\begin{align*}
		\frac{1}{T}\sum_{t=1}^{T} \EE\Big[f(x_t) - f(x^*)\Big]
		\leq&
		\frac{1}{T}\sum_{t=1}^{T} \EE \Big[f_\alpha(x_t) - f_\alpha(x^*)\Big] + L_0\alpha\\
		\leq& 
		\frac{\norm{x_1 - x^*}^2}{T\eta} + 24\eta c_0 dL_0^2 + \frac{24\eta d^2\sigma_0^2}{\alpha^2} + L_0\alpha.
	\end{align*}
\end{proof}

\begin{proof}[Proof of Corollary~\ref{cor:main_ns}]
	First, we will check that the step size $\eta$ set as Eq.~\eqref{eq:eta_alp} satisfies the condition $\eta \leq \alpha/(3dL_0)$ required in Theorem~\ref{thm:main_nss}.
	We have
	\begin{align}
		\eta = \frac{\norm{x_1 - x^*} \alpha}{(24dT)^{1/2} (c_0 L_0^2\alpha^2 + d\sigma_0^2)^{1/2}}
		\leq
		\frac{\norm{x_1 - x^*} \alpha}{(24dT)^{1/2} c_0^{1/2} L_0\alpha}
		\leq \alpha/(3dL_0),
	\end{align}
	where the last inequality is because of the condition $T \ge \frac{3 L_0^2 \norm{x_1 - x^*}^2}{2^{11} \cdot c_0^2\sigma_0^2 }$.

	
	By Eq.~\eqref{eq:eta_alp}, Eq.~\eqref{eq:nss} reduces to 
	\begin{align*}
		\frac{1}{T}\sum_{t=1}^{T} \Big(f(x_t) - f(x^*)\Big)
		\leq&
		\frac{2(24d)^{1/2}(c_0 L_0^2\alpha^2 + d\sigma_0^2)^{1/2} \norm{x_1 - x^*} }{\alpha T^{1/2}} + L_0\alpha\\
		\leq&
		\frac{2(24d)^{1/2}(c_0^{1/2} L_0\alpha + d^{1/2}\sigma_0) \norm{x_1 - x^*} }{\alpha T^{1/2}} + L_0\alpha\\
		=&
		\frac{2(96)^{1/4}d^{1/2}\sigma_0^{1/2}L_0^{1/2} \norm{x_1 - x^*}^{1/2}}{ T^{1/4}} + \frac{2(24d)^{1/2}c_0^{1/2} L_0 \norm{x_1 - x^*}}{T^{1/2}},
	\end{align*}
	By setting two terms in the right hand of above equation to $\varepsilon/2$, respectively, we can obtain that to find an $\varepsilon$-suboptimal solution, it  requires that 
	\begin{align*}
		T =& \max\left(\frac{16\times 96d^2 \sigma_0^2L_0^2\norm{x_1 - x^*}^2}{\varepsilon^4},\frac{96dc_0L_0^2 \norm{x_1 - x^*}^2}{\varepsilon^2}\right)\\
		=&96L_0^2\norm{x_1 - x^*}^2\cdot\max\left(\frac{16 d^2 \sigma_0^2}{\varepsilon^4},\frac{dc_0}{\varepsilon^2}\right).
	\end{align*}
	
\end{proof}

\begin{proof}[Proof of Theorem~\ref{thm:main_ns2}]
	Similar to Eq.~\eqref{eq:x4}, we have
	\begin{align*}
		\EE\left[\norm{x_{t+1} - x^*}^2\right]
		=&
		\norm{x_t - x^*}^2 - \eta \EE\left[\dotprod{\tg_t, x_t - x^*}\right] + \eta^2 \EE\left[\norm{\tg_t}^2\right]\\
		=&
		\norm{x_t - x^*}^2 - \eta \dotprod{\nabla f_\alpha(x_t), x_t - x^*} + \eta^2 \EE\left[\norm{\tg_t}^2\right]\\
		\leq&
		\left(1 - \frac{\mu\eta}{2}\right)\norm{x_t - x^*}^2 - \eta \Big(f_\alpha(x_t) - f_\alpha(x^*)\Big) + \eta^2\EE\left[\norm{\tg_t}^2\right]\\
		\stackrel{\eqref{eq:srecs}}{\leq}&
		\left(1 - \frac{\mu\eta}{2}\right)\norm{x_t - x^*}^2 - \eta \Big(f_\alpha(x_t) - f_\alpha(x^*)\Big) 
		+\left(24c_0 d L_0^2 + \frac{24d^2\sigma_0^2}{\alpha^2}\right)\cdot \eta^2,
	\end{align*}
	where the last inequality is because of $f_\alpha(x)$ is $\mu$-strongly convex by Lemma~\ref{lem:f_alp}.

	Using the notation $\rho$, and we represent above equation as follows:
	\begin{align*}
		\EE\Big[f_\alpha(x_t) - f_\alpha(x^*)\Big]
		\leq
		\frac{\rho \norm{x_t - x^*}^2}{\eta} - \frac{\norm{x_{t+1} - x^*}^2}{\eta} + \left(24c_0 d L_0^2 + \frac{24d^2\sigma_0^2}{\alpha^2}\right)\cdot \eta.
	\end{align*} 
	Thus, we have
	\begin{align*}
		&\sum_{t=1}^{T} w_t\EE\Big[f_\alpha(x_t) - f_\alpha(x^*)\Big]\\
		\leq& 
		\sum_{t=1}^{T}\left(\frac{w_t\rho \EE\norm{x_t - x^*}^2}{\eta} - \frac{w_t\EE\norm{x_{t+1} - x^*}^2}{\eta}\right) 
		+ \left(24c_0 d L_0^2 + \frac{24d^2\sigma_0^2}{\alpha^2}\right)\cdot \eta \sum_{t=1}^{T}w_t\\
		=&
		\sum_{t=1}^{T}\left(\frac{w_{t-1} \EE\norm{x_t - x^*}^2}{\eta} - \frac{w_t\EE\norm{x_{t+1} - x^*}^2}{\eta}\right) 
		+ \left(24c_0 d L_0^2 + \frac{24d^2\sigma_0^2}{\alpha^2}\right)\cdot \eta \sum_{t=1}^{T}w_t\\
		=&
		\frac{w_0 \norm{x_1 - x^*}^2}{\eta} - \frac{w_T\EE\norm{x_{T+1} - x^*}^2}{\eta}
		+ \left(24c_0 d L_0^2 + \frac{24d^2\sigma_0^2}{\alpha^2}\right)\cdot \eta \sum_{t=1}^{T}w_t\\
		\leq&
		\frac{w_0 \norm{x_1 - x^*}^2}{\eta} 
		+ \left(24c_0 d L_0^2 + \frac{24d^2\sigma_0^2}{\alpha^2}\right)\cdot \eta \sum_{t=1}^{T}w_t,
	\end{align*}
	where the first equality is because of $w_t = \rho^{-1} w_{t-1}$ and the last equality is because of telescoping terms.
	
	Combining with the facts $f_\alpha(x) \ge f(x)$ and $f_\alpha(x) \leq f(x) + L_0\alpha$ by Lemma~\ref{lem:sf}, we can obtain that
	\begin{align*}
		\frac{1}{\sum_{t=1}^{T} w_t}\sum_{t=1}^{T} w_t\EE\left[f(x_t) - f(x^*)\right]
		\leq&
		\frac{1}{\sum_{t=1}^{T} w_t}\sum_{t=1}^{T} w_t\EE\left[f_\alpha(x_t) - f_\alpha(x^*)\right] + L_0\alpha\\
		\leq&
		\frac{w_0\norm{x_t - x^*}^2}{\eta \sum_{t=1}^{T} w_t}
		+
		\left(24c_0 d L_0^2 + \frac{24d^2\sigma_0^2}{\alpha^2}\right)\cdot\eta + L_0\alpha\\
		=&
		\frac{\mu \norm{x_1 - x^*}^2}{2\left(\rho^{-T} - 1\right)}
		+ 
		\left(24c_0 d L_0^2 + \frac{24d^2\sigma_0^2}{\alpha^2}\right)\cdot\eta + L_0\alpha,
	\end{align*}
	where the first equality is because of 
	\begin{align*}
		\sum_{t=1}^{T}w_t 
		= \frac{\rho^{-1} \left(\rho^{-T} - 1\right)}{\rho^{-1} - 1}
		= \frac{ \left(\rho^{-T} - 1\right)}{1 - \rho}
		= \frac{ \left(\rho^{-T} - 1\right)}{\mu\eta/2}.
	\end{align*}
	
	Finally, by the convexity of $f(x)$, we can obtain that
	\begin{align*}
		\EE\left[f\left(\frac{\sum_{t=1}^{T}w_t x_t}{\sum_{t=1}^{T}w_t}\right) - f(x^*)\right] 
		\leq 
		\frac{1}{\sum_{t=1}^{T} w_t}\sum_{t=1}^{T} w_t\EE\left[f(x_t) - f(x^*)\right].
	\end{align*}
	Combining above results, we can obtain the final result.
\end{proof}

\begin{proof}[Proof of Corollary~\ref{cor:main_ns2}]
	First, we will check that $\eta = \frac{L_0\alpha^3}{24d^2\sigma_0^2}$ satisfies $\eta \leq \frac{\alpha}{3dL_0}$ which is required by Theorem~\ref{thm:main_ns2}.
	Equivalently, we only need to check that $\alpha^2 \leq \frac{8d\sigma_0^2}{L_0^2}$ and we have
	\begin{align*}
		\alpha^2 
		=
		\min\left\{ \frac{\varepsilon^2}{8^2L_0^2},  \left(\frac{d\sigma_0^2\varepsilon}{4c_0 L_0^3}\right)^{2/3}\right\}  
		\leq \frac{8d\sigma_0^2}{L_0^2}, 
	\end{align*}
	where the last inequality is because of the condition $\varepsilon <  4(8d(8c_0+1))^{1/2}\cdot \sigma_0$.
	
	By setting $\eta = \frac{L_0\alpha^3}{24d^2\sigma_0^2}$, Eq.~\eqref{eq:ns_s2} will reduce to
	\begin{align*}
		\EE\left[f\left(\frac{\sum_{t=1}^{T}w_t x_t}{\sum_{t=1}^{T}w_t}\right) - f(x^*)\right] 
		\leq& 
		\frac{\mu \norm{x_1 - x^*}^2}{2\left(\rho^{-T} - 1\right)}
		+ 2L_0\alpha + \frac{c_0 L_0^3\alpha^3}{d\sigma_0^2}\\
		\leq& 
		\frac{\mu \norm{x_1 - x^*}^2}{2\left(\rho^{-T} - 1\right)} + \frac{\varepsilon}{4} + \frac{\varepsilon}{4}
		=
		\frac{\mu \norm{x_1 - x^*}^2}{2\left(\rho^{-T} - 1\right)} + \frac{\varepsilon}{2},
	\end{align*}
	where the second inequality is because of $\alpha = \min\left\{\frac{\varepsilon}{8L_0},\left(\frac{d\sigma_0^2\varepsilon}{4c_0 L_0^3}\right)^{1/3}\right\}$.
	Furthermore, to achieve $\frac{\mu\norm{x_1 - x^*}^2}{2(\rho^{-T} - 1)} \leq \frac{\varepsilon}{2}$, we only require that $\rho^{-T} \geq \frac{\mu\norm{x_1 - x^*}^2}{\varepsilon} + 1$ which equals to $\rho^T \leq  \left(\frac{\mu\norm{x_1 - x^*}^2}{\varepsilon} + 1\right)^{-1}$.
	By the value of $T$ in Eq.~\eqref{eq:sns2}, we can validate that
	\begin{align*}
		\rho^T  = (1-\mu\eta)^T 
		\leq 
		\exp\left(-T\mu\eta\right)
		=
		\exp\left(-\log\left(\frac{\mu\norm{x_1 - x^*}^2}{\varepsilon} + 1\right)\right) 
		= \left(\frac{\mu\norm{x_1 - x^*}^2}{\varepsilon} +  1\right)^{-1},
	\end{align*}
	which concludes the proof.

\end{proof}

\section{Proofs of Section~\ref{subsec:scvx} }

\begin{proof}[Proof of Lemma~\ref{lem:sgt_recs1}]
	First, we represent Eq.~\eqref{eq:sgt_def} as follows:
	\begin{equation}\label{eq:sgts}
		\begin{aligned}
			\tg_t 
			=& d\cdot \frac{f_\alpha(x_t+\alpha u_t,\; \xi_t) - f_\alpha(x_{t-1} + \alpha u_{t-1}, \;\xi_{t-1})}{\alpha} u_t 
			+ d\cdot\frac{f(x_t + \alpha u_t,\;\xi_t) - f_\alpha(x_t+\alpha u_t,\;\xi_t)}{\alpha}u_t\\
			&
			+ d\cdot\frac{ f_\alpha(x_{t-1}+\alpha u_{t-1},\; \xi_{t-1}) - f(x_{t-1} + \alpha u_{t-1},\; \xi_{t-1}) }{\alpha}u_t.
		\end{aligned}
	\end{equation}
	
	For the notation convenience, we denote that
	\begin{align*}
		\tg_t' = d\cdot \frac{f_\alpha(x_t+\alpha u_t,\; \xi_t) - f_\alpha(x_{t-1} + \alpha u_{t-1},\;\xi_{t-1})}{\alpha} u_t.
	\end{align*}
	By the Taylor's expansion, we have 
	\begin{align*}
		f_\alpha(x_t + \alpha u_t, \; \xi_t) =& f_\alpha(x_t, \; \xi_t) + \alpha \dotprod{\nabla_\alpha f(x_t, \; \xi_t), u_t} + \phi(x_t, u_t,\alpha, \; \xi_t),\\
		f_\alpha(x_{t-1} + \alpha u_{t-1},\; \xi_{t-1}) =& f_\alpha(x_{t-1},\; \xi_{t-1}) + \alpha \dotprod{\nabla f_\alpha(x_{t-1},\; \xi_{t-1}), u_{t-1}} + \phi(x_{t-1}, u_{t-1},\alpha,\; \xi_{t-1}),
	\end{align*} 
	where $\phi(x_t, u_t, \alpha, \; \xi_t) \triangleq f_\alpha(x_t + \alpha u_t, \; \xi_t) - \Big(f_\alpha(x_t, \; \xi_t) + \alpha \dotprod{\nabla_\alpha f(x_t, \; \xi_t), u_t}\Big) $.
	Thus, we can obtain that
	\begin{align*}
		\tg_t' 
		=& d\cdot \left(u_tu_t^\top\nabla f_\alpha(x_t,\; \xi_t) - u_tu_{t-1}^\top\nabla f_\alpha(x_{t-1},\; \xi_{t-1})\right)\\
		& 
		+ d\cdot \frac{ f_\alpha(x_t,\; \xi_t) - f_\alpha(x_{t-1},\; \xi_{t-1})}{\alpha}u_t 
		+ d\cdot \frac{\phi(x_t, u_t,\alpha,\; \xi_t) - \phi(x_{t-1}, u_{t-1},\alpha,\; \xi_{t-1})}{\alpha}u_t.
	\end{align*}
	Accordingly, we have
	\begin{equation}\label{eq:s4}
		\begin{aligned}
			&\EE\left[\norm{\tg_t'}^2 \right]
			\leq
			4d^2 \EE\left[ \norm{u_tu_t^\top \nabla f_\alpha(x_t,\; \xi_t)}^2 \right] + 4d^2 \EE\left[  \Big( u_{t-1}^\top \nabla f_\alpha(x_{t-1},\; \xi_{t-1}) \Big)^2\right] \cdot \EE\left[ \norm{u_t}^2 \right] \\
			&+ \frac{4d^2\EE\left[(f_\alpha(x_t,\; \xi_t) - f_\alpha(x_{t-1},\; \xi_{t-1}))^2\right]}{\alpha^2} \EE\left[\norm{u_t}^2\right] \\
			&+ \EE\left[\frac{4d^2\Big(\phi(x_t, u_t,\alpha,\; \xi_t) - \phi(x_{t-1}, u_{t-1},\alpha,\; \xi_{t-1})  \Big)^2}{\alpha^2} \norm{u_t}^2 \right].
		\end{aligned}
	\end{equation}
	Furthermore, 
	\begin{align*}
		&\EE\left[(f_\alpha(x_t,\; \xi_t) - f_\alpha(x_{t-1},\; \xi_{t-1}))^2\right] \\
		\leq& 
		4\EE_{\xi_t}\left[(f_\alpha(x_t,\; \xi_t) - f_\alpha(x_t))^2\right]
		+ 4\EE_{\xi_{t-1}}\left[(f_\alpha(x_{t-1},\; \xi_{t-1}) - f_\alpha(x_{t-1}))^2\right]
		+ 2 \Big(f_\alpha(x_t) - f_\alpha(x_{t-1})\Big)^2\\
		\leq&
		8\sigma_0^2 + 2 \Big(f_\alpha(x_t) - f_\alpha(x_{t-1})\Big)^2,
	\end{align*}
	where the last inequality is because of Assumption~\ref{ass:bnd_val_var}.
	By the $L_0$-Lipschitz continuous of $f(x)$ and Lemma~\ref{lem:f_alp}, we can obtain that
	\begin{align*}
		\Big(f_\alpha(x_t) - f_\alpha(x_{t-1})\Big)^2 
		\leq
		L_0^2\norm{x_t - x_{t-1}}^2 = L_0^2\eta^2 \norm{\tg_{t-1}}^2. 
	\end{align*}

	By Eq.~\eqref{eq:uuas}, we can obtain that
	\begin{equation*}
		\EE_{u_t}\left[\norm{u_tu_t^\top \nabla f_\alpha(x_t,\; \xi_t)}^2\right] = \frac{1}{d}\norm{\nabla f_\alpha(x_t,\; \xi_t)}^2, \mbox{ and } \EE_{u_t}\left[\norm{u_t}^2\right] = 1.
	\end{equation*} 
	
	Furthermore, it holds that 
	\begin{align*}
		\EE_{u_{t-1}}\left[\Big(u_{t-1}^\top \nabla f_\alpha(x_{t-1},\; \xi_{t-1})\Big)^2\right] = \frac{1}{d}\norm{\nabla f_\alpha(x_{t-1},\; \xi_{t-1})}^2.
	\end{align*}
	Thus, we can obtain that
	\begin{align*}
		&\EE\left[  \norm{u_tu_t^\top \nabla f_\alpha(x_t,\; \xi_t)}^2\right] 
		= \frac{1}{d}\EE\left[  \norm{\nabla f_\alpha(x_t,\; \xi_t)}^2 \right] \\
		\leq& 
		\frac{2}{d} \left(\norm{\nabla f_\alpha(x_t)}^2 + \EE\left[\norm{\nabla f_\alpha(x_t) - \nabla f_\alpha(x_t,\;\xi_t)}^2\right]\right) 
		\leq
		\frac{2}{d} \left( \norm{\nabla f_\alpha(x_t)}^2 + \sigma_1^2 \right),
	\end{align*}
	where the last inequality is because of Assumption~\ref{ass:B_SG}.
	Similarly, we can obtain that
	\begin{equation*}
		\EE\left[\Big(u_{t-1}^\top \nabla f_\alpha(x_{t-1},\; \xi_{t-1})\Big)^2\right]
		\leq 	\frac{2}{d} \left( \norm{\nabla f_\alpha(x_{t-1})}^2 + \sigma_1^2 \right)
	\end{equation*}
	
	Since $f(x,\;\xi)$ is $L$-smooth, we can conclude that $f_\alpha(x,\;\xi)$ is also $L$-smooth by Lemma~\ref{lem:f_alp}.
	Thus, by Eq.~\eqref{eq:L}, we can obtain that
	\begin{equation*}
		|\phi(x_t, u_t,\alpha,\xi_t)| \leq \frac{L\alpha^2}{2}\norm{u_t}^2 = \frac{L\alpha^2}{2}.
	\end{equation*}
	Accordingly, we have
	\begin{align*}
		\EE_{u_t}\left[\Big(\phi(x_t, u_t,\alpha,\;\xi_t) - \phi(x_{t-1}, u_{t-1},\alpha,\;\xi_{t-1})  \Big)^2 \norm{u_t}^2  \right]
		\leq
		\frac{L^2\alpha^4}{2}
		+
		\frac{L^2\alpha^4}{2}
		=
		L^2\alpha^4.
	\end{align*}
	Combining above results with Eq.~\eqref{eq:s4} and using full expectation rule, we can obtain that
	\begin{align*}
		\EE\left[\norm{\tg_t'}^2\right]
		\leq&
		8d\left(\norm{\nabla f_\alpha(x_t)}^2 + \sigma_1^2\right)	
		+ 8d\left(\norm{\nabla f_\alpha(x_{t-1})}^2 + \sigma_1^2\right)	\\
		& + \frac{4d^2 (8\sigma_0^2 + 2 L_0^2 \eta^2 \norm{\tg_{t-1}}^2)}{\alpha^2} 
		+ 4d^2 L^2 \alpha^2.
	\end{align*}
	Combining with Eq.~\eqref{eq:sgts}, we can obtain that
	\begin{align*}
		\EE\left[\norm{\tg_t}^2\right]
		\leq&
		2\EE\left[\norm{\tg_t'}^2\right]
		+ 4d^2 \EE\left[\frac{(f_\alpha(x_t + \alpha u_t, \;\xi_t) - f(x_t + \alpha u_t, \;\xi_t))^2 }{\alpha^2}\norm{u_t}^2  \right]\\
		\\ &+ 4d^2 \EE\left[\frac{ (f_\alpha(x_{t-1} + \alpha u_{t-1},\; \xi_{t-1}) - f(x_{t-1} + \alpha u_{t-1},\; \xi_{t-1}))^2}{\alpha^2}\norm{u_t}^2  \right]\\
		\leq&
		2\EE\left[\norm{g_t'}^2\right]
		+ 2d^2L^2\alpha^2\\
		\leq&
		16d\norm{\nabla f_\alpha(x_t)}^2
		+ 16d \norm{\nabla f_\alpha(x_{t-1})}^2
		+ \frac{16d^2L_0^2\eta^2}{\alpha^2} \norm{\tg_{t-1}}^2
		+ \frac{64d^2\sigma_0^2}{\alpha^2} + 32d\sigma_1^2
		+ 10d^2L^2\alpha^2\\
		\leq&
		16d\norm{\nabla f_\alpha(x_t)}^2
		+ 16d \norm{\nabla f_\alpha(x_{t-1})}^2
		+ \frac{1}{4} \norm{\tg_{t-1}}^2
		+ \frac{64d^2\sigma_0^2}{\alpha^2} + 32d\sigma_1^2
		+ 10d^2L^2\alpha^2,
	\end{align*}
	where the second inequality is because of $|f(x)-f_\alpha(x)| \leq \frac{L\alpha^2}{2}$ by Lemma~\ref{lem:sf} and the last inequality is because of $\eta \leq \frac{\alpha}{8dL_0} $.
	
\end{proof}

\begin{proof}[Proof of Theorem~\ref{thm:main_scvx}]
	Similar to Eq.~\eqref{eq:xx}, we can obtain that
	\begin{align*}
		\EE\left[\norm{x_{t+1} - x^*}^2\right]
		\leq
		\norm{x_t - x^*}^2 - \eta \Big(f_\alpha(x_t) - f_\alpha(x^*)\Big) - \frac{\eta}{2L}\norm{\nabla f_\alpha(x_t)}^2 + \eta^2\EE\left[\norm{\tg_t}^2\right].
	\end{align*}
	Representing  above equation and summing it from $t=1$ to $T$, we can obtain that
	\begin{align*}
		&\eta \sum_{t=1}^{T} \EE\Big[f_\alpha(x_t) - f_\alpha(x^*)\Big] \\
		\leq& 
		\sum_{t=1}^{T} \left[\EE\norm{x_t - x^*}^2 - \EE\norm{x_{t+1} - x^*}^2\right]
		-
		\frac{\eta}{2L}\sum_{t=1}^{T} \norm{\nabla f_\alpha(x_t)}^2
		+
		\eta^2\sum_{t=1}^{T}\EE\left[\norm{\tg_t}^2\right]\\
		=& 
		\norm{x_1 - x^*}^2 - \EE\norm{x_{T+1} - x^*}^2 
		-
		\frac{\eta}{2L}\sum_{t=1}^{T} \norm{\nabla f_\alpha(x_t)}^2
		+
		\eta^2\sum_{t=1}^{T}\EE\left[\norm{\tg_t}^2\right].
	\end{align*}
	
	By Eq.~\eqref{eq:srecs} and using Lemma~\ref{lem:Delta} and Lemma~\ref{lem:ss} with $\rho = 1/4$ and $M_t = 16d(\norm{\nabla f_\alpha(x_t)}^2 + \norm{\nabla f_\alpha(x_{t-1})}^2)+ \frac{64d^2\sigma_0^2}{\alpha^2} + 32d\sigma_1^2
	+ 10d^2L^2\alpha^2$, we can obtain that
	\begin{equation}\label{eq:sgt_sum}
		\begin{aligned}
			\sum_{t=1}^{T}	\EE\left[\norm{\tg_t}^2\right] 
			\leq&
			\sum_{t=1}^{T}  \left[\left(\frac{1}{4}\right)^t \norm{\tg_0}^2 + \sum_{i=0}^{t-1} \left(\frac{64d^2\sigma_0^2}{\alpha^2} + 32d\sigma_1^2
			+ 10d^2L^2\alpha^2\right) \frac{1}{4^i}  \right]\\
			&+
			16d\sum_{t=1}^{T}\sum_{i=0}^{t-1}\left(\norm{\nabla f_\alpha(x_t)}^2 + \norm{\nabla f_\alpha(x_{t-1})}^2\right)\frac{1}{4^i} \\
			\leq&
			32d \sum_{t=1}^{T} \norm{\nabla f_\alpha(x_t)}^2 + 16d\norm{\nabla f_\alpha(x_0)}^2 + \frac{128d^2\sigma_0^2 T}{\alpha^2} + 64d\sigma_1^2T  +  20d^2L^2\alpha^2 T
		\end{aligned}
	\end{equation}
	where the last inequality is also because of  $\norm{\tg_0} = 0$ implied by  $x_1 = x_0$.
	
	Thus, we can obtain that
	\begin{align*}
		\frac{1}{T}\sum_{t=1}^{T}\EE\Big[f_\alpha(x_t) - f_\alpha(x^*)\Big]
		\leq&
		\frac{\norm{x_1 - x^*}^2}{\eta T}
		- \frac{1}{2LT} \sum_{t=1}^{T} \norm{\nabla f_\alpha(x_t)}^2 + \frac{32d\eta}{T} \sum_{t=1}^{T} \norm{\nabla f_\alpha(x_t)}^2\\
		&
		+ 20 \eta d^2L^2\alpha^2 + \frac{128d^2\sigma_0^2 \eta}{\alpha^2} + 64d\eta\sigma_1^2 
		+ \frac{16\eta d}{T} \norm{\nabla f_\alpha(x_0)}^2\\
		\leq&
		\frac{\norm{x_1 - x^*}^2}{\eta T}
		+ 20 \eta d^2L^2\alpha^2 + \frac{128d^2\sigma_0^2 \eta}{\alpha^2} + 64d\eta\sigma_1^2 
		+ \frac{16\eta d}{T} \norm{\nabla f_\alpha(x_0)}^2\\
		\leq&
		\frac{\norm{x_1 - x^*}^2}{\eta T}
		+ 20 \eta d^2L^2\alpha^2
		+ \frac{16\eta d L_0^2}{T}  + \frac{128d^2\sigma_0^2 \eta}{\alpha^2} + 64d\eta\sigma_1^2, 
	\end{align*}
	where the second inequality is because of $\eta \leq \frac{1}{64 dL}$ and the last inequality is because  $f(x)$ is $L_0$-Lipschitz continuous.
	
	Combining with the fact $f_\alpha(x) \geq f(x)$ and $f_\alpha(x^*)\leq f(x^*) + \frac{L\alpha^2}{2}$ by Lemma~\ref{lem:sf}, we can obtain that
	\begin{align*}
		\frac{1}{T}\sum_{t=1}^{T}\EE\Big[f(x_t) - f(x^*)\Big]
		\leq&
		\frac{1}{T}\sum_{t=1}^{T}\EE\Big[f_\alpha(x_t) - f_\alpha(x^*)\Big] + \frac{L\alpha^2}{2}\\
		\leq&
		\frac{\norm{x_1 - x^*}^2}{\eta T}
		+ 20 \eta d^2L^2\alpha^2
		+ \frac{16\eta d L_0^2}{T} 
		+ \frac{L\alpha^2}{2}+ \frac{128d^2\sigma_0^2 \eta}{\alpha^2} + 64d\eta\sigma_1^2.
	\end{align*}
\end{proof}

\begin{proof}[Proof of Corollary~\ref{cor:main_scvx}]
	By Eq.~\eqref{eq:alp_eta_scvx}, we can obtain that
	\begin{align*}
		\eta = \frac{\norm{x_1 - x^*}^{4/3}}{2 L^{1/3}d^{2/3}\sigma_0^{2/3} T^{2/3}}.
	\end{align*}
	Then, by $T \geq \max\left\{ \frac{L_0^2 \norm{x_1 - x^*}^2}{\sigma_0^2}, \; \frac{2^{7/2} \cdot L d^{1/2} \norm{x_1 - x^*}^2}{\sigma_0} \right\}$ shown in Eq.~\eqref{eq:T_rq_scvx}, we can obtain that our step size satisfies $\eta \leq \min\left\{\frac{\alpha}{8dL_0}, \frac{1}{64 dL}\right\}$ required in Theorem~\ref{thm:main_scvx}.
	
	Furthermore, by $T \geq \frac{20^{3/2} L\norm{x_1 - x^*}^2}{ \sigma_0 d^4}$, we can obtain that $20 \eta d^2 L^2 \alpha^2 \leq \frac{L\alpha^2}{2}$. 
	Thus, Eq.~\eqref{eq:dec_snc} will reduce to 
	%
	\begin{align*}
		\frac{1}{T}\sum_{t=1}^{T}\EE\Big[f(x_t) - f(x^*)\Big]
		\leq 
		\frac{\norm{x_1 - x^*}^2}{\eta T}
		+ \frac{16\eta d L_0^2}{T} 
		+ L\alpha^2 
		+ \frac{128d^2\sigma_0^2 \eta}{\alpha^2} 
		+ 64d\eta\sigma_1^2.		
	\end{align*}
	Replacing Eq.~\eqref{eq:alp_eta_scvx} to above equation, we can obtain that
	\begin{align*}
		\frac{1}{T}\sum_{t=1}^{T}\EE\Big[f(x_t) - f(x^*)\Big]
		\leq& 
		\frac{8\norm{x_1 - x^*} d\sigma_0}{\alpha \sqrt{T}} 
		+ \frac{2\alpha L_0^2 \norm{x_1 - x^*}}{\sigma_0 T^{3/2}} 
		+ L\alpha^2 
		+ \frac{8\alpha \sigma_1^2 \norm{x_1 - x^*}}{\sqrt{T}\sigma_0}\\
		\leq& 
		\frac{6 L^{1/3} (d \norm{x_1 - x^*}\sigma_0)^{2/3}}{T^{1/3}} 
		+ \frac{8d^{1/3}\norm{x_1 - x^*}^{4/3} L_0^2}{L^{1/3}\sigma_0^{2/3} T^{5/3}} 
		+ \frac{32d^{1/3}\sigma_1^2 \norm{x_1 - x^*}^{4/3}}{L^{1/3}\sigma_0^{2/3}T^{2/3}}\\
		\leq&
		\frac{6 L^{1/3} (d \norm{x_1 - x^*}\sigma_0)^{2/3}}{T^{1/3}} 
		+ \frac{64d^{1/3}\sigma_1^2 \norm{x_1 - x^*}^{4/3}}{L^{1/3}\sigma_0^{2/3}T^{2/3}},
	\end{align*}
	where the last inequality is because of $ \frac{8d^{1/3}\norm{x_1 - x^*}^{4/3} L_0^2}{L^{1/3}\sigma_0^{2/3} T^{5/3}} 
	\leq 
	\frac{32d^{1/3}\sigma_1^2 \norm{x_1 - x^*}^{4/3}}{L^{1/3}\sigma_0^{2/3}T^{2/3}} $ if $T \geq \frac{L_0^2}{4\sigma_1^2}$. 
	
	Setting two terms in the right hand of above equation to $\frac{\varepsilon}{2}$ separately, then we can obtain that
	\begin{align*}
		T = \max\left\{ \frac{12^3 \cdot L d^2 \norm{x_1 - x^*}^2 \sigma_0^2}{\varepsilon^3},\; \frac{2^{21/2} \cdot d^{1/2} \norm{x_1 - x^*}^2 \sigma_1^3 }{L^{1/2} \sigma_0 \varepsilon^{3/2}}\right\},
	\end{align*} 
	which concludes the proof.
	
\end{proof}

\begin{proof}[Proof of Theorem~\ref{thm:main_sncvx}]
	The objective function $f(x)$ satisfies the properties required in Lemma~\ref{lem:tg_ass} and Lemma~\ref{lem:AA}.
	Eq.~\eqref{eq:srecs} shows that the gradient estimation $\tg_t$  satisfies the properties required in Lemma~\ref{lem:tg_ass} with $c_1 = 8$ and $\zeta = \frac{32d^2\sigma_0^2}{\alpha^2} + 16d\sigma_1^2
	+ 5d^2L^2\alpha^2$.
	Furthermore, the step size $\eta \leq \min\left\{\frac{\alpha}{4dL_0}, \frac{1}{112dL}\right\}$ satisfies that $\eta \leq \frac{1}{14c_1 dL}$ required in Lemma~\ref{lem:AA}. 
	Thus, results of Lemma~\ref{lem:tg_ass} and Lemma~\ref{lem:AA} hold.
	
	Then, by using Eq.~\eqref{eq:recs2} recursively and combining with Lemma~\ref{lem:Delta}, we can obtain that
	\begin{equation}\label{eq:dd_rec}
		\begin{aligned}
			&\rho \Delta_t + \frac{1}{2} \cdot \frac{L\eta^2}{2} \EE\left[ \norm{\tg_{t-1}}^2\right] + c_1dL\eta^2 \norm{\nabla f_\alpha(x_{t-1})}^2 + A_i\zeta L\eta^2 \\
			\leq&
			\rho^t \Delta_1 + \rho^{t-1} \frac{L\eta^2}{2} \frac{A_{t-1}}{2} \norm{\tg_0}^2 + c_1\rho^{t-1}\cdot A_{t-1}dL\eta^2 \norm{\nabla f_\alpha(x_0)}^2
			+\zeta \sum_{i=0}^{t} \rho^i A_i 
			\\
			\leq&
			\rho^t \Delta_1 + \rho^{t-1} \cdot L\eta^2 \norm{\tg_0}^2 + 3c_1\rho^{t-1}\cdot dL\eta^2 \norm{\nabla f_\alpha(x_0)}^2
			+ \frac{3\zeta L\eta^2}{1-\rho},
		\end{aligned}
	\end{equation}
	where the last inequality is because of $A_i \leq 3$ and $\sum_{i=0}^{t-1}\rho^iA_i \leq 3 \sum_{i=0}^{t-1}\rho^i \leq 
	\frac{3}{1-\rho}$.
	
	Therefore,
	\begin{align*}
		\EE\left[f_\alpha(x_{t+1}) - f_\alpha(\tx^*)\right]
		\stackrel{\eqref{eq:dec}}{\leq}&
		\rho \Delta_t + \frac{1}{2} \cdot \frac{L\eta^2}{2} \EE\left[ \norm{\tg_{t-1}}^2\right] + c_1dL\eta^2 \norm{\nabla f_\alpha(x_{t-1})}^2 + A_0\zeta L\eta^2\\
		\stackrel{\eqref{eq:dd_rec}}{\leq}&
		\rho^t \Delta_1 + \rho^{t-1} \cdot L\eta^2 \norm{\tg_0}^2 + 3c_1\rho^{t-1}\cdot dL\eta^2 \norm{\nabla f_\alpha(x_0)}^2
		+ \frac{3\zeta L\eta^2}{1-\rho}.
	\end{align*}
	Combining above results and using the fact $\norm{\tg_0} = 0$ and $\norm{\nabla f_\alpha(x)}^2 \leq L_0^2$, we can obtain that 
	\begin{align}
		\EE\left[ f_\alpha(x_{t+1}) - f_\alpha(\tx^*) \right]
		\leq& 
		\rho^t \Big(f_\alpha(x_1) - f_\alpha(\tx^*)\Big)
		+ 24\rho^{t-1}\cdot dLL_0^2\eta^2  
		+ \frac{3L\eta^2\zeta}{1-\rho}.
	\end{align}
	Combining with facts that $f(x) \leq f_\alpha(x) \leq f(x) + \frac{L\alpha^2}{2}$ and $f_\alpha(\tx^*) \geq f(x^*) - \frac{L\alpha^2}{2}$ shown in Lemma~\ref{lem:sf}, we can obtain that
	\begin{align*}
		&\EE\left[f(x_{t+1}) - f(x^*)\right]
		\leq
		\rho^t\Big(f(x_1) - f(x^*) + L\alpha^2\Big)
		+ 24\rho^{t-1}\cdot dLL_0^2\eta^2  
		+ \frac{3L\eta^2\zeta}{1-\rho}
		+ L\alpha^2\\
		\leq& \rho^t\Big(f(x_1) - f(x^*)\Big)
		+ 24\rho^{t-1}\cdot dLL_0^2\eta^2  
		+ \frac{3L\eta}{\mu}\left(\frac{32d^2\sigma_0^2}{\alpha^2} + 16d\sigma_1^2
		+ 5d^2L^2\alpha^2\right)
		+ 2L\alpha^2.
	\end{align*}
	
	%
\end{proof}

\begin{proof}[Proof of Corollary~\ref{cor:main_sncvx}]
	By setting $\eta = \frac{\mu \alpha^4}{48d^2\sigma_0^2}$, we can obtain that
	\begin{align*}
		\EE\left[f(x_{t+1}) - f(x^*)\right]
		\leq&
		\rho^t\Big(f(x_1) - f(x^*) \Big)
		+ 24\rho^{t-1}\cdot dLL_0^2\eta^2  
		+ \frac{L\sigma_1^2\alpha^4}{d\sigma_0^2} 
		+ \frac{5L^3\alpha^6}{16\sigma_0^2}
		+ 4L\alpha^2.
	\end{align*}
	If $\alpha^2 \leq \frac{2\sigma_0}{L}$, then it holds that $ \frac{5L^3\alpha^6}{16\sigma_0^2} \leq 4L\alpha^2 $. 
	Accordingly, we can obtain that
	\begin{equation*}
		\EE\left[f(x_{t+1}) - f(x^*)\right]
		\leq
		\rho^t\Big(f(x_1) - f(x^*) \Big)
		+ 24\rho^{t-1}\cdot dLL_0^2\eta^2  
		+ \frac{L\sigma_1^2\alpha^4}{d\sigma_0^2} 
		+ 8L\alpha^2.
	\end{equation*}
	
	By setting $8L\alpha^2 \leq  \frac{\varepsilon}{4}$ and $\frac{L\sigma_1^2\alpha^4}{d\sigma_0^2} \leq \frac{\varepsilon}{4}$, we can obtain that $	\alpha^2 \leq  \frac{\varepsilon}{32L} $ and $\alpha^4 \leq \frac{d\sigma_0^2 \varepsilon}{4L\sigma_1^2} $.
	Thus, we can obtain that 
	\begin{align*}
		\EE\left[f(x_{t+1}) - f(x^*)\right]
		\leq& 
		\rho^t\Big(f(x_1) - f(x^*) \Big)
		+ 24\rho^{t-1}\cdot dLL_0^2\eta^2
		+\frac{\varepsilon}{2}\\
		\leq& \frac{\varepsilon}{4} + \frac{\varepsilon}{4} + \frac{\varepsilon}{2} 
		= \varepsilon,
	\end{align*}
	where the last inequality is because of value of $T$.
\end{proof}

\section{Proofs of Section~\ref{subsec:sncvx}}
\begin{proof}[Proof of Theorem~\ref{thm:main_snc}]
	Since $f(x, \xi)$ is $L$-smooth, then $f(x)$ is also $L$-smooth. 
	Accordingly, we can conclude that $f_\alpha(x)$ is  $L$-smooth by Lemma~\ref{lem:f_alp}.
	Let $\tx^*$ denote the minimal point of $f_\alpha(x)$.
	Then, by the update of Algorithm~\ref{alg:SCG}, we can obtain that
	\begin{equation*}
		\begin{aligned}
			\EE\left[f_\alpha(x_{t+1}) \right]
			\leq& 
			f_\alpha(x_t)  - \eta \EE\left[\dotprod{\nabla f_\alpha(x_t), \tg_t}\right] + \frac{L\eta^2}{2}\EE\left[\norm{\tg_t}^2\right]\\
			=&
			f_\alpha(x_t) -  f_\alpha(\tx^*) 
			- \eta \norm{\nabla f_\alpha(x_t)}^2
			+ \frac{L\eta^2}{2}\EE\left[\norm{\tg_t}^2\right].
		\end{aligned}
	\end{equation*}
	We reformulate above equation as 
	\begin{equation*}
		\eta \EE\left[\norm{\nabla f_\alpha(x_t)}^2\right] 
		\le
		\EE\left[f_\alpha(x_t)\right] - \EE\left[f_\alpha(x_{t+1})\right] + \frac{L\eta^2}{2}\EE\left[\norm{\tg_t}^2\right]. 
	\end{equation*}
	Summing up above equation from $1$ to $T$ and telescoping terms, we can obtain that
	\begin{align*}
		\eta\sum_{t=1}^{T}\EE\left[\norm{\nabla f_\alpha(x_t)}^2\right]&\leq
		f_\alpha(x_1) - \EE\left[f_\alpha(x_{T+1})\right] + \frac{L\eta^2}{2}\sum_{t=1}^{T}\EE \left[\norm{\tg_t}^2\right]\\
		&\stackrel{\eqref{eq:sgt_sum}}{\leq}
		f_\alpha(x_1) - \EE\left[f_\alpha(x_{T+1})\right]
		+ 16dL\eta^2 \sum_{t=1}^{T} \norm{\nabla f_\alpha(x_t)}^2
		+ 8dL\eta^2 \norm{\nabla f_\alpha(x_0)}^2 \\
		&+ \frac{64d^2\sigma_0^2\eta^2 T}{\alpha^2}
		+ 32d\sigma_1^2\eta^2 T
		+ 10d^2 \eta^2L^3\alpha^2T\\
		&\leq
		f_\alpha(x_1) - f_\alpha(\tx^*)
		+ 16dL\eta^2 \sum_{t=1}^{T} \norm{\nabla f_\alpha(x_t)}^2
		+ 8dLL_0^2\eta^2  
		\\
		&+ \frac{64d^2\sigma_0^2\eta^2 T}{\alpha^2}
		+ 32d\sigma_1^2\eta^2 T
		+ 10d^2 \eta^2L^3\alpha^2T.
	\end{align*} 
	The above equation can be represented as
	\begin{align*}
		\eta(1 - 16dL\eta) \sum_{t=1}^{T}\EE\left[\norm{\nabla f_\alpha(x_t)}^2\right] 
		\leq 
		f_\alpha(x_1) - f_\alpha(\tx^*)
		+ 8dLL_0^2\eta^2  
		+ \frac{64d^2\sigma_0^2\eta^2 T}{\alpha^2}
		+ 32d\sigma_1^2\eta^2 T
		+ 10d^2 \eta^2L^3\alpha^2T.
	\end{align*}
	Furthermore, by $\eta \leq \frac{1}{32dL}$, it holds that $\frac{\eta}{2} \leq \eta(1 - 16dL\eta) $.
	Thus,
	\begin{align*}
		\frac{1}{T}\sum_{t=1}^{T}\EE\left[\norm{\nabla f_\alpha(x_t)}^2\right] 
		&= 
		\frac{2}{\eta T} \cdot \frac{\eta}{2} \sum_{t=1}^{T}\EE\left[\norm{\nabla f_\alpha(x_t)}^2\right]\\
		&\leq
		\frac{2}{\eta T} \cdot \eta(1 - 16dL\eta) \sum_{t=1}^{T}\EE\left[\norm{\nabla f_\alpha(x_t)}^2\right]\\ 
		&\leq
		\frac{2\left(f_\alpha(x_1) - f_\alpha(\tx^*)\right)}{\eta T} 
		+ \frac{16dLL_0^2 \eta}{T}
		+ 20 \eta d^2 L^3 \alpha^2 + \frac{128d^2\sigma_0^2\eta }{\alpha^2}
		+ 64d\sigma_1^2\eta.
	\end{align*}
	By Lemma~\ref{lem:sf}, we can obtain that
	\begin{align*}
		\norm{\nabla f(x_t)}^2 
		=&
		\norm{\nabla f_\alpha(x_t) + \nabla f(x_t) - \nabla f_\alpha(x_t)}^2\\
		\leq&
		2\norm{\nabla f_\alpha(x_t)}^2 + 2 \norm{\nabla f(x_t) - \nabla f_\alpha(x_t)}^2\\
		\leq&
		2\norm{\nabla f_\alpha(x_t)}^2 + 2L^2\alpha^2. 
	\end{align*}
	Also by Lemma~\ref{lem:sf}, we have
	\begin{align*}
		f_\alpha(x_1) - f_\alpha(x_{T+1})
		\leq 
		f_\alpha(x_1) - f_\alpha(\tx^*)
		\leq
		f(x_1) + \frac{L\alpha^2}{2} - f(x^*) + \frac{L\alpha^2}{2}
		=
		f(x_1) - f(x^*) + L\alpha^2.
	\end{align*}
	
	Combining above results, we can obtain that
	\begin{align*}
		&\frac{1}{T}\sum_{t=1}^{T}\EE\left[\norm{\nabla f(x_t)}^2\right]
		\leq
		2\cdot \frac{1}{T}\sum_{t=1}^{T}\EE\left[\norm{\nabla f_\alpha(x_t)}^2\right] + 2L^2\alpha^2\\
		\leq&
		\frac{4\left(f_\alpha(x_1) - f_\alpha(\tx^*)\right)}{\eta T} 
		+ \frac{32dLL_0^2 \eta}{T}
		+ 40 \eta d^2 L^3 \alpha^2
		+ 2L^2\alpha^2	+  \frac{128d^2\sigma_0^2\eta }{\alpha^2}
		+ 64d\sigma_1^2\eta\\
		\leq&
		\frac{4\left(f(x_1) - f(x^*)\right)}{\eta T} 
		+ \frac{4L\alpha^2}{\eta T}
		+ \frac{32dLL_0^2 \eta}{T}
		+ 40 \eta d^2 L^3 \alpha^2
		+ 2L^2\alpha^2	+  \frac{128d^2\sigma_0^2\eta }{\alpha^2}
		+ 64d\sigma_1^2\eta.
	\end{align*}
\end{proof}

\begin{proof}[Proof of Corollary~\ref{cor:main_snc}]
	By setting $\eta = \frac{\alpha (f(x_1) - f(x^*))^{1/2} }{4\sqrt{T}d\sigma_0}$,
	we can obtain that
	\begin{align*}
		&\frac{1}{T}\sum_{t=1}^{T}\EE\left[\norm{\nabla f(x_t)}^2\right]
		\leq
		\frac{  33d\sigma_0(f(x_1) - f(x^*))^{1/2}}{\alpha \sqrt{T}}
		+ \frac{8\alpha LL_0^2(f(x_1) - f(x^*))^{1/2}}{\sigma_0 T^{3/2}}\\
		&+  \frac{10d\alpha^3L^3(f(x_1) - f(x^*))^{1/2}}{\sqrt{T}\sigma_0}
		+ 2L^2\alpha^2 + \frac{16\alpha\sigma_1^2(f(x_1) - f(x^*))^{1/2}}{\sigma_0\sqrt{T}}
		+ \frac{  16dL \alpha\sigma_0 }{ \sqrt{T} (f(x_1) - f(x^*))^{1/2}}
	\end{align*}
	
	Setting $\alpha = \frac{(f(x_1) - f(x^*))^{1/6} (d\sigma_0)^{1/3}}{L^{2/3}T^{1/6}}$,
	We can obtain that
	\begin{align*}
		&\frac{1}{T}\sum_{t=1}^{T}\EE\left[\norm{\nabla f(x_t)}^2\right]
		\leq
		\frac{35(d\sigma_0)^{2/3} L^{2/3} (f(x_1) - f(x^*))^{1/3}}{T^{1/3}}
		+\frac{8d^{1/3}(f(x_1) - f(x^*))^{2/3}L^{1/3}L_0^2}{T^{5/3}\sigma_0^{2/3}}\\
		&+\frac{10Ld^2(f(x_1) - f(x^*))}{T}
		+ \frac{16d^{1/3}(f(x_1) - f(x^*))^{2/3}\sigma_1^2}{L^{2/3}\sigma_0^{2/3}T^{2/3}}
		+ \frac{  16(d\sigma_0)^{4/3}L^{1/3}  }{ T^{2/3} (f(x_1) - f(x^*))^{1/3}}\\
		&\leq
		\frac{35(d\sigma_0)^{2/3} L^{2/3} (f(x_1) - f(x^*))^{1/3}}{T^{1/3}}
		+\frac{10Ld^2(f(x_1) - f(x^*))}{T}\\
		&
		+ \frac{16d^{1/3}(f(x_1) - f(x^*))^{2/3}\sigma_1^2}{L^{2/3}\sigma_0^{2/3}T^{2/3}}
		+ \frac{  24(d\sigma_0)^{4/3}L^{1/3}  }{ T^{2/3} (f(x_1) - f(x^*))^{1/3}},
	\end{align*}
	where the last inequality is because of $T \geq \frac{(f(x_1) - f(x^*)) L_0^2}{d\sigma_0^2}$.
	By setting four terms in the right hand of above equation to $\varepsilon^2/4$ respectively, we can obtain that to achieve $\frac{1}{T}\sum_{t=1}^{T}\EE\left[\norm{\nabla f(x_t)}^2\right] \leq \varepsilon^2$, we only require that
	\begin{equation*}
		T = (f(x_1) - f(x^*)) \cdot \max\left\{ \frac{140^3 \cdot  d^2 L^2 \sigma_0^2}{\varepsilon^6}, \frac{40\cdot Ld^2}{\varepsilon^2}, \frac{2^9 \cdot d^{1/2} \sigma_1^3}{L\sigma_0 \varepsilon^3}, \frac{96^{3/2} \cdot d^2\sigma_0^2 L^{1/2}}{\varepsilon^3 (f(x_1) - f(x^*))^{3/2}}\right\}.
	\end{equation*}
\end{proof}
\end{document}